\documentclass[a4paper,10pt]{article}

\usepackage{geometry}
\usepackage[english]{babel}
\usepackage{epsfig}
\usepackage{amsmath}
\usepackage{amssymb}
\usepackage{amsthm}
\usepackage{mathrsfs}
\usepackage{color,amscd}
\usepackage{currfile}
\usepackage{hyperref}

\let\oldcolor\color
\renewcommand{\color}[1]{\oldcolor{#1}}  

\usepackage{amsthm}
\newtheorem{theorem}{Theorem}[section]

\newtheorem{lemma}[theorem]{Lemma}
\newtheorem{proposition}[theorem]{Proposition}
\newtheorem{corollary}[theorem]{Corollary}
\newtheorem{remark}[theorem]{Remark}
\newtheorem{example}[theorem]{Example}
\newtheorem{examples}[theorem]{Examples}
\newtheorem{notation}[theorem]{Notation}

\usepackage{amsfonts}

\newcommand{\oo}{{\mathbb{O}}}
\newcommand{\hh}{{\mathbb{H}}}
\newcommand{\HH}{{\mathbb{H}}}

\newcommand{\cc}{{\mathbb{C}}}
\newcommand{\rr}{{\mathbb{R}}}
\newcommand{\zz}{{\mathbb{Z}}}
\newcommand{\nn}{{\mathbb{N}}}
\newcommand{\pp}{{\mathbb{P}}}
\newcommand{\jj}{{\mathbb{J}}}

\newcommand{\s}{{\mathbb{S}}}

\newcommand{\sr}{\mathcal{SR}}
\newcommand{\SC}{\mathcal{SC}}
\newcommand{\n}{\mathcal{N}}

\newcommand{\I}{\mathcal{I}}

\newcommand\sd[1]{{#1}'_s}

\newcommand\re{\operatorname{Re}}
\newcommand\im{\operatorname{Im}}

\newcommand{\ui}{\imath}

\newcommand{\OO}{\Omega}
\newcommand{\Cl}{\mr{Cl}}

\newcommand{\mbb}{\mathbb}

\newcommand{\mr}{\mathrm}
\newcommand{\mscr}{\mathscr}
\newcommand{\R}{\mbb{R}}

\newcommand{\cd}{\frac{\partial f}{\partial x}}

\newcommand{\dd[3]}{\frac{\partial #2}{\partial #3}}

\newcommand{\dcf}{\overline{\partial}_{\scriptscriptstyle CRF}}

\newcommand{\sss}{\scriptscriptstyle}

\title{\bf On a class of orientation-preserving maps of $\R^4$}

\vspace{1em} 

\author{
\textsc{Riccardo Ghiloni and Alessandro Perotti}}

\date{
{\small
\textit{Department of Mathematics, University of Trento, I--38123, Povo-Trento, Italy}} 
\\
{\small \texttt{ghiloni@science.unitn.it} $\qquad$ \texttt{perotti@science.unitn.it}}
}

\begin{document}

\maketitle


\begin{abstract}
The purpose of this paper is to present several new, sometimes surprising, results concerning a class of hyperholomorphic functions over quaternions, the so-called slice regular functions. The concept of slice regular function is a generalization of the one of holomorphic function in one complex variable. The results we present here show that such a generalization is multifaceted and highly non-trivial. We study the behavior of the Jacobian matrix $J_f$ of a slice regular function $f$ proving in particular that $\det(J_f)\geq0$, i.e. $f$ is orientation-preserving. We give a complete characterization of the fibers of $f$ making use of a new notion we intro\-duce here, the one of wing of $f$. We investigate the singular set $N_f$ of $f$, i.e. the set in which $J_f$ is singular. The singular set $N_f$ turns out to be equal to the branch set of $f$, i.e. the set of points $y$ such that $f$ is not a homeomorphism locally at $y$. We establish the quasi-openness properties of $f$. As a consequence we deduce the validity of the Maximum Modulus Principle for $f$ in its full generality. Our results are sharp as we show by explicit examples.

\vspace{.5em}


\vspace{.5em}

\end{abstract}



\section{Introduction}

Holomorphic functions of a complex variable are of central importance in mathematics. The deep interplay between the analytic nature and the algebraic nature of these functions is one of their peculiar features, which makes possible their applications to many areas of science.

Holomorphy is equivalent to complex analyticity, in particular complex polynomials are holomorphic. On a domain $D$ of $\cc$ the fibers of a non-constant holomorphic function $f$ are discrete, in particular $f$ admits a holomorphic reciprocal out of a discrete set, provided $f\not\equiv0$. These are examples of basic properties having relevant analytic and algebraic consequences for holomorphic functions. A remarkable differential characteristic of holomorphic functions is represented by the equality between the determinant of the Jacobian matrix of $f$ and the squared norm of its complex derivative, i.e. $\det(J_f)=|f'|^2$. As a consequence $f$ is orientation-preserving, i.e. $\det(J_f)\geq0$. In addition $y$ is a branch point of $f$ if and only if $y$ is a singular point of $f$, i.e. $f$ is not a homeomorphism locally at $y$ if and only if the Jacobian matrix $J_f(y)$ is singular. The holomorphic function $f$ is also an open map, independently from the presence of branch points. Consequently, $f$ satisfies the Maximum Modulus Principle.

These are a few fundamental results of the theory of holomorphic functions of a complex variable, which one would desire to have in a generalization of this theory in dimension higher than two. During the last century several generalizations were introduced mainly over quaternions $\HH$, octonions $\oo$ and Clifford algebras $\R_m$, see \cite{BDS,GHS}. These generalized theories share many analytic and differential behaviors with the theory of holomorphic functions of a complex variable. However they do not include the classical theory of polynomials with noncommutative coefficients on one side, see \cite{lam}.

In 2006 Gentili and Struppa \cite{GeSt2006CR,GeSt2007Adv} remedied to this `algebraic' lack introducing a new theory, the one of slice regular functions over quaternions. Such a theory was extended to octonions and Clifford algebras in \cite{CoSaSt2009Israel,GeSt2008COV,GeStRocky}. In paper \cite{GhPe_AIM} we gave a unified and generalized approach valid over all real alternative $^*$-algebras, based on the concept of stem function. The theory has developed rapidly, see e.g. \cite{GeStoSt2013,DivisionAlgebras} and references therein. It has proven also its effectiveness in applications to quaternionic functional calculus and mathematical foundation of quaternionic quantum mechanics (see e.g.\  \cite{libroverde,GMP,GMP2017}), classification of orthogonal complex structures in $\R^4$ (see e.g.\ \cite{Altavilla2018,AltavillaSarfatti,TwistorJEMS}), and operator semigroup theory in noncommutative setting (see e.g.\ \cite{CoSaADVMATH_2011,GR_2016,GR_2018}).

The stem function approach over quaternions reads as follows, see \cite{GhPe_AIM}. 

Let $\HH$ be the real division algebra of quaternions and let $\s_\HH=\{J\in\HH\ |\ J^2=-1\}$ be the $2$-sphere of its imaginary units. For each $J\in\s_\HH$, we denote by $\cc_J=\mr{Span}(1,J)\simeq\cc$ the subalgebra of $\HH$ generated by $J$. Then we have the `slice' decomposition
\[
\HH=\bigcup_{J\in \s_\HH}\cc_J,\text{\ where $\cc_J\cap\cc_K=\R$ for every $J,K\in\s_\HH$ with $J\ne\pm K$}.
\]
Given a (non-empty) subset $D$ of $\cc$, we define the circularization $\OO_D$ of $D$ as follows:
\[
\OO_D:=\{\alpha+J\beta \in \HH \,:\, \alpha,\beta\in\rr, \alpha+i\beta\in D, J\in\s_\HH\}.
\]
If $x=\alpha+J\beta\in\cc_J$ and $z:=\alpha+i\beta\in\cc$, then $\OO_{\{z\}}$ is denoted by $\s_x$ and it is equal to the $2$-sphere $\alpha+\beta\s_\HH$ if $x\not\in\rr$ and to the singleton $\{x\}$ if $x\in\rr$. A subset $S$ of $\HH$ is said to be circular if it is equal to $\OO_D$ for some $D\subset\cc$. This is equivalent to say that $\s_x\subset S$ for each $x\in S$.

\emph{In what follows we assume $D$ is an open subset of $\cc$, invariant under the complex conjugation}.

Consider the complexified algebra $\HH\otimes_\rr\cc=\{x+\ui y\,|\,x,y\in\HH\}$ of $\HH$, endowed with the product $(x+\ui y)(x'+\ui y'):=(xx'-yy')+\ui(xy'+yx')$, so $\ui^2=-1$. A function $F=F_1+\ui F_2:D\to\HH\otimes_\rr\cc$ is called \emph{stem function} if $F$ is even-odd w.r.t. $\beta$ in the sense that $F_1(\overline{z})=F_1(z)$ and $F_2(\overline{z})=-F_2(z)$ for each $z\in D$. Note that the function $F=F_1+\ui F_2$ is holomorphic if $F_1$ and $F_2$ are of class $\mscr{C}^1$ and 
$\frac{\partial F}{\partial\beta}=\ui\frac{\partial F}{\partial\alpha}$, i.e. $\frac{\partial F_1}{\partial\alpha}=\frac{\partial F_2}{\partial\beta}$ and $\frac{\partial F_1}{\partial\beta}=-\frac{\partial F_2}{\partial\alpha}$. 

Consider now the circular open subset $\OO_D$ of $\HH$. A function $f:\OO_D\to\HH$ is called \emph{(left) slice regular function} if there exists a holomorphic stem function $F=F_1+\ui F_2:D\to\HH\otimes_\rr\cc$ such that, for each $z=\alpha+i\beta\in D$, for each $J\in\s_\HH$ and for each $x=\alpha+J\beta\in\OO_D$,
\[
f(x)=F_1(z)+JF_2(z).
\]
If this is the case, we say that $f$ is induced by $F$ and we write $f=\I(F)$. The even-odd character of $F$ ensures the coherence of  definition of $f=\I(F)$. Moreover, the slice regular function $f$ is induced by a unique  stem function $F$.

Denote by $\sr(\OO_D)$ the real vector space of all slice regular functions on $\OO_D$. The pointwise product $FG$ of two holomorphic stem functions $F$ and $G$ is again a holomorphic stem function. This allows to define the \emph{slice product} of $f=\I(F)$ and $g=\I(G)$ in $\sr(\OO_D)$ by $f\cdot g:=\I(FG)$. The slice product makes $\sr(\OO_D)$ a real algebra. Such an algebra preserves derivatives in the following sense: the complex derivative $\frac{\partial F}{\partial z}=\frac{\partial F}{\partial\alpha}$ of $F$ is again a holomorphic stem function, which induces the element $\frac{\partial f}{\partial x}=\I(\frac{\partial F}{\partial z})$ of $\sr(\OO_D)$, called \emph{slice derivative} of $f$. As we have just mentioned, a remarkable novelty of slice regularity theory is that the real algebra $\sr(\HH)$ contains as a subalgebra the algebra of polynomials with quaternionic coefficients on the right. Indeed, such a polynomial $p(x)=\sum_{n=0}^dx^na_n$ is induced by the polynomial stem function $P(\alpha+i\beta)=\sum_{n=0}^d(\alpha+\ui\beta)^na_n$. Furthermore, the classical product of $p(x)$ with another polynomial $q(x)=\sum_{n=0}^ex^nb_n$, obtained by imposing commutativity of the indeterminate with the coefficients, coincides exactly with the slice product $(p\cdot q)(x)=\sum_{n=0}^{d+e}x^n(\sum_{m+\ell=n}a_mb_\ell)$. 

Slice regular functions are real analytic in the usual real sense, see \cite[Proposition 7(3)]{GhPe_AIM}. Slice regularity is equivalent to spherical analyticity, a new concept introduced in \cite{StoppatoAdvMath2012} (see also \cite{singular}). In particular, locally at a real point a function $f$ is slice regular if and only if it admits a quaternionic series expansion of the form $\sum_{n\in\nn}x^na_n$.

These facts reveal the algebraic and analytic relevance of slice regular functions.

The nature of the domain of definition $\OO_D$ is important in the study of slice regular functions. If $\OO_D$ is connected and intersects the real line $\R$, then it is said to be a \emph{slice domain}. If $\OO_D$ is connected and does not intersect $\R$ then it is called \emph{product domain}. In the first case, $D$ is connected (it is a domain of $\cc$) and intersects $\R$. In the second, $D$ does not intersect $\rr$ and has two connected components $D^+$ and $D^-$, switched by complex conjugation; moreover, $\OO_D$ is homeomorphic to the topological product $\s_\HH\times D^+$. In general $\OO_D$ decomposes into the disjoint union of its connected components, which are slice domains or product domains. Then, in the study of slice regular functions, we can always assume $\OO_D$ is a slice domain or a product domain. It is important to recall that, if $\OO_D$ is a slice domain, then the notion of slice regular function $f$ we give above coincides with the original one introduced by Gentili and Struppa \cite{GeSt2006CR,GeSt2007Adv}; indeed $f$ is slice regular if and only if, for each $J\in\s_\HH$, the restriction of $f$ to $\OO_D\cap\cc_J$ is holomorphic w.r.t. the complex structures defined by the left multiplication by $J$.   

In spite of rapid development of the theory and of its applications, until now, some basic features of slice regular functions $f:\OO_D\to\HH$ are not well understood yet, as the nature of the fibers of $f$ in the case $\OO_D$ is a product domain and for general $\OO_D$ the behavior of the Jacobian of $f$ and the resulting structure of branch set of $f$.

In this paper we give quite an exhaustive study of these basic features. New concepts are introduced. Several new, sometimes surprising, phenomena appear; they describe a manifold geometric scenario, showing how the slice regularity theory is a wealthy and highly non-trivial generalization of holomorphic function theory in one complex variable.

The results presented here are sharp in the sense that we are able to give explicit examples for all the geometric configurations predicted by the results. We prove that the determinant of the Jacobian matrix of $f$ can be expressed in terms of the squared norm of a suitable Hermitian product between the slice derivative $\frac{\partial f}{\partial x}$ of $f$ and another kind of derivative of $f$, the so-called \emph{spherical derivative} $f'_s$ of $f$. As a consequence $f$ is orientation-preserving. We investigate deeply the structure of fibers of $f$. We introduce the brand new notion of \emph{wing} of $f$. If $\OO_D$ is a product domain and $\jj_D$ is the complex structure on $\OO_D$ sending $\alpha+J\beta$ with $\beta>0$ into the left multiplication by $J$, then a wing of $f$ is a complex analytic curve of $(\OO_D,\jj_D)$ on which $f$ is constant. We denote by $W_f$ the union of all wings of $f$ and give an explicit analytic criterion to detect whether $W_f$ is empty or not. We~establish the existence of exactly eight distinct situations in which $W_f$ can be either empty or formed by a unique wing, two wings or a $\s^1$-fibration of wings. Combining the above positivity property of the Jacobian of $f$, the above properties of the fibers of $f$ and some fine results from differential topology, we are able to prove a completely new `Quasi-open Mapping Theorem' for $f$ defined on general $\OO_D$. As a consequence we deduce the Maximum Modulus Principle for $f$ in its full generality. The techniques developed in the paper permit to show that the branch set of $f$ is equal to the singular set of $f$ as in the classical holomorphic case. Furthermore, the singular locus $N_f$ of $f$ decomposes into three subsets, the zero set $D_f$ of spherical derivative $f'_s$, the set $W_f\setminus D_f$ and their complement $N_f\setminus(D_f\cup W_f)$. Let $d_f=\dim(D_f)$, $w_f=\dim(W_f\setminus D_f)$ and $m_f=\dim(N_f\setminus(D_f\cup W_f))$. We prove that the triple $(d_f,w_f,m_f)$ can assume exactly five values when $\OO_D$ is a slice domain and eleven when $\OO_D$ is a product domain. Finally we give one more application of the positivity of the Jacobian of $f$, by presenting a quaternionic counterpart of a classical boundary univalence criterion valid for holomorphic functions of a complex variable.

The paper is organized as follows. In Section \ref{sec:preliminaries} we recall briefly some preliminary material. Sections \ref{signJacobian} and \ref{sec:formulaJacobian} deal with the sign of the Jacobian of $f$, and some explicit formulas for $J_f$. In Sections \ref{sec:fibers} and \ref{sec:singularset} we study the fibers, the singular set and the quasi-openness of $f$. Section \ref{sec:univalence} concerns the mentioned boundary univalence criterion for $f$.


\section{Preliminaries}\label{sec:preliminaries}

Let us recall some basic material concerning slice functions over $\HH$, see \cite{GhPe_AIM} and also \cite{GhPe_Trends,AlgebraSliceFunctions,singular} for a full account presented via the stem function approach, including generalizations.

Consider a (non-empty) open subset $D$ of $\cc$ invariant under complex conjugation. Let us generalize the notion of slice regular function introduced above. A function $f:\OO_D\to\HH$ is a \emph{(left) slice function} if there exists an arbitrary (not necessarily holomorphic) stem function $F=F_1+\ui F_2:D\to\HH\otimes_\rr\cc$ such that, for each $z=\alpha+i\beta\in D$, for each $J\in\s_\HH$ and for each $x=\alpha+J\beta\in\OO_D$, it holds
\[
f(x)=F_1(z)+JF_2(z).
\]
We say that $f$ is induced by $F$ and we write $f=\I(F)$. Note that the definition of $f$ is coherent. Indeed, if $x\in\OO_D\cap\rr$ then $z=x$, $F_2(x)=0$ and hence $f(x)=F_1(x)$. If $x\in\OO_D\setminus\rr$ then there exist, and are unique, $\alpha,\beta\in\rr$ and $J\in\s_\HH$ such that $\beta>0$ and $x=\alpha+J\beta=\alpha+(-J)(-\beta)$ so $F_1(z)+JF_2(z)=f(x)=F_1(\overline{z})+(-J)F_2(\overline{z})$, where $z=\alpha+i\beta\in D$. The slice function $f$ is induced by a unique stem function $F=F_1+\ui F_2$. Indeed, for each fixed $J\in\s_\HH$, if $x=\alpha+J\beta\in\OO_D$ and $z=\alpha+i\beta\in D$, then $F_1(z)=\frac{1}{2}(f(x)+f(\overline{x}))$ and $F_2(z)=-\frac{J}{2}(f(x)-f(\overline{x}))$, where $\overline{x}$ denotes the standard quaternionic conjugation of $x$. The latter fact implies that the slice function $f$ satisfies the following \emph{representation formula}: if $I\in\s_\HH$ and $y=\alpha+I\beta\in\OO_D$, then
\[
f(y)=\frac{1}{2}(f(x)+f(\overline{x}))-\frac{I}{2}(J(f(x)-f(\overline{x}))).
\]
As a consequence, the slice function $f$ is uniquely determined by its values on $\OO_D\cap\cc_J$. Moreover, $f$ is affine on each $2$-sphere $\s_x$ if $x\not\in\rr$. We denote by $\mathcal{S}(\OO_D)$ the real vector space of all slice functions on $\OO_D$, endowed with the standard pointwise defined operations.

Let $f=\I(F),g=\I(G)\in\mathcal{S}(\OO_D)$, with $F=F_1+\ui F_2$ and $G=G_1+\ui G_2$. Evidently, the pointwise product $FG=(F_1G_1-F_2G_2)+\ui(F_1G_2+F_2G_1)$ is again a stem function. In this way we can define the \emph{slice product $f\cdot g\in\mathcal{S}(\OO_D)$} of $f$ and $g$ by setting $f\cdot g:=\I(FG)$. This product makes $\mathcal{S}(\OO_D)$ a real algebra, containing $\sr(\OO_D)$ as a subalgebra.

Given $J\in\s_\HH$, the slice function $f=\I(F_1+\ui F_2)$ is called \emph{$\cc_J$-slice-preserving} if $F_1$ and $F_2$ are $\cc_J$-valued. Note that the slice function $f$ is $\cc_J$-slice-preserving if and only if $f(\OO_D\cap\cc_J)\subset\cc_J$. We denote by $\sr_{\cc_J}(\OO_D)$ the set of all $\cc_J$-slice-preserving slice regular functions on $\OO_D$, which turns out to be a subalgebra of $\sr(\OO_D)$. We say that $f$ is \emph{one-slice-preserving} if it is $\cc_J$-slice preserving for some $J\in\s_\HH$ and we denote by $\sr_\cc(\OO_D)$ the set of all one-slice-preserving slice regular functions, i.e. $\sr_\cc(\OO_D)=\bigcup_{J\in\s_\HH}\sr_{\cc_J}(\OO_D)$.

The slice function $f=\I(F_1+\ui F_2)$ is called \emph{slice-preserving} if it is $\cc_J$-slice-preserving for each $J\in\s_\HH$. This is equivalent to require that $F_1$ and $F_2$ are real-valued. Note that if $f$ is slice-preserving, the slice product $f\cdot g$ coincides with the pointwise product $fg$ for any slice function $g$. The same is true if $f$ is any slice function and $G=G_1+\ui G_2$ takes values in $\HH$, i.e. $G_2\equiv0$. We denote by $\sr_\rr(\OO_D)$ the set of all slice-preserving slice regular functions on $\OO_D$, which coincides with the subalgebra $\bigcap_{J\in\s_\HH}\sr_{\cc_J}(\OO_D)$ of $\sr(\OO_D)$.

Consider again the slice function $f=\I(F)\in\mathcal{S}(\OO_D)$, with $F=F_1+\ui F_2$. Denote by $F^c:D\to\HH\otimes_\rr\cc$ the stem function defined by $F^c(z):=\overline{F_1(z)}+\ui\overline{F_2(z)}$ for every $z\in D$. The slice function $f^c\in\mathcal{S}(\OO_D)$ induced by $F^c$ is called \emph{conjugate function} of $f$ and the slice function $N(f):=f\cdot f^c\in\mathcal{S}(\OO_D)$ is said to be the \emph{normal function} of $f$. Some authors use the term \emph{symmetrization} $f^s$ of $f$ instead of normal function $N(f)$ of $f$. It is immediate to verify that $N(f)$ is slice-preserving, and $f^c$ is slice regular if and only if $f$ is. In particular, if $f\in\sr(\OO_D)$ then $N(f)\in\sr_\rr(\OO_D)$.

The slice function $f=\I(F)$ is called \emph{slice-constant} if $F$ is locally constant on $D$. This is equivalent to say that $f$ is slice regular and its slice derivative $\cd$ (see the Introduction for the definition) vanishes identically on $\OO_D$. We denote by $\SC(\OO_D)$ the subalgebra of $\sr(\OO_D)$ formed by slice-constant slice functions on $\OO_D$. 

Given a quaternion $y$, we denote by $\re(y)$ the real part of $y$ and by  $\im(y)$ the imaginary part of $y$, i.e. $\re(y)=\frac{1}{2}(y+\overline{y})\in\rr$ and $\im(y)=y-\re(y)=\frac{1}{2}(y-\overline{y})$. The normal function of the polynomial $x\mapsto x-y$ is called \emph{characteristic polynomial} of $y$. It is denoted by $\Delta_y$ and it holds:
\[
\Delta_y(x)=x^2-2\re(y)x+|y|^2,
\]
where $|y|$ is the Euclidean norm of $y\in\HH\simeq\rr^4$. The zero set of $\Delta_y$ coincides with $\s_y$.

Let us recall the definition of spherical derivative of $f=\I(F_1+\ui F_2)$. The function $D\setminus\rr\to\HH\otimes_\rr\cc$, sending $\alpha+i\beta$ into $\frac{F_2(\alpha+i\beta)}{\beta}$, is $\beta$-even so it is a stem function, which takes values in~$\HH$. The slice function on $\OO_D\setminus\rr$ induced by such a stem function is denoted by $f'_s$ and it is called \emph{spherical derivative} of $f$. The zero set $D_f$ of $\sd f$ is called \emph{degenerate set} of $f$. It is immediate to verify that $f$ is constant on some $\s_x$ with $x\not\in\rr$ if and only if $\sd f(x)=0$ or equivalently $\s_x\subset D_f$. Furthermore, $\sd f(x)=\frac{1}{2}\im(x)^{-1}(f(x)-f(x^c))$ for every $x\in\OO_D$. As a consequence, if $f\in\sr(\OO_D)$, then $\sd f$ is real analytic; hence its zero set $D_f$ is closed and real analytic in $\OO_D\setminus\rr$. 

We denote by $V(f)$ the zero set of $f=\I(F)$, i.e. $V(f)=\{x\in\OO_D\,:\,f(x)=0\}$. A quite important fact is that $V(N(f))=\bigcup_{y\in V(f)}\s_y$. If $\OO_D$ is connected, $f\in\sr(\OO_D)$ and $N(f)\not\equiv0$, then the elements $x$ of $V(f)$ can be of three types: \emph{real zeros} of $f$ if $x\in\R$, \emph{spherical zeros} of $f$ if $x\not\in\R$ and $\s_x\subset V(f)$, or \emph{isolated non-real zeros} of $f$ if $\s_x\not\subset V(f)$.

Finally, we recall the notion of total multiplicity of a zero of $f\in\sr(\OO_D)$ we will use in Propositions \ref{pro:Delta} and \ref{pro:branch} below. A point $y$ in $\OO_D$ belongs to $V(f)$ if and only if $\Delta_y$ divides $N(f)$ in $\sr(\OO_D)$. Suppose that $y\in V(f)$ and $f\not\equiv0$ on the connected component of $y$ in $\OO_D$. Given a non-negative integer $s$, we say that $y$ is a zero of $f$ of \emph{total multiplicity} $s$ if $\Delta_y^s$ divides $N(f)$ and $\Delta_y^{s+1}$ does not divide $N(f)$ in $\sr(\OO_D)$. We denote such a non-negative integer $s$ by $m_f (y)$.

\begin{notation}
Throughout the remaining part of the paper, we assume that $\OO_D$ is a slice domain or a product domain. Moreover, for simplicity, we often use $\OO$ instead of $\OO_D$.
\end{notation}


\section{Sign of the Jacobian}\label{signJacobian}

Let $f=\I(F_1+\ui F_2)\in\sr(\OO_D)$. Suppose $\OO_D$ is a slice domain. Since $F_2$ is $\beta$-odd and real analytic, there exists a (unique) $\beta$-even real analytic function $\widehat{F}_2:D\to\hh$ such that $\widehat{F}_2(\alpha,\beta)=\frac{F_2(\alpha,\beta)}{\beta}$ on $D\setminus\R$. Note that $\widehat{F}_2=\frac{\partial F_2}{\partial\beta}=\frac{\partial F_1}{\partial\alpha}=\cd$ on $D\cap\R$. Proposition 7(3) of \cite{GhPe_AIM} ensures that $\I(\widehat{F}_2)$ is a real analytic function. In particular it is the unique continuous extension of $f'_s$ on $\OO_D$. We define the \emph{extended spherical derivative $\widehat{f'_s}$ of $f$} as $\widehat{f'_s}:=\I(\widehat{F}_2)$. Note that $\widehat{f'_s}=f'_s$ on $\OO_D\setminus\R$ and $\widehat{f'_s}=\cd$ on $\OO_D\cap\R$. If $\OO_D$ is a product domain, then we set $\widehat{F}_2:=F_2$ and $\widehat{f'_s}:=f'_s$.

\begin{proposition}\label{pro:matrix}
Let $f\in\sr(\OO)$, let $y\in\OO\cap\cc_I$ and let $J\in\s_\hh$ be orthogonal to $I$. Let $\cd(y)=q_0+q_1I+q_2J+q_3IJ$ and $J\widehat{f'_s}(y)=p_0+p_1I+p_2J+p_3IJ$, with $q_i,p_i\in\R$ for $i=0,1,2,3$. Then the real differential $df_y$ of $f$ at $y$ is represented w.r.t.\ the basis $\{1,I,J,IJ\}$ of $\hh\simeq\R^4$ by the matrix 
\begin{equation}\label{eq:matrix}
J_f(y)=
\left[
\begin{array}{rrrr}
q_0&-q_1&p_0&-p_1\\
q_1&q_0&p_1&p_0\\
q_2&-q_3&p_2&-p_3\\
q_3&q_2&p_3&p_2
\end{array}
\right].
\end{equation}

\end{proposition}
\begin{proof}
Let $x=x_0+x_1I+x_2J+x_3IJ\in\OO$. Let $f=\I(F)$. Since $\cd=\I\left(\frac{\partial F}{\partial z}\right)$, the Cauchy-Riemann equations satisfied by $F$ give
\[\frac{\partial f}{\partial x_0}(y)=\cd(y)
\quad \text{and}\quad \frac{\partial f}{\partial x_1}(y)=I\frac{\partial f}{\partial x_0}(y)=I\cd(y)=-q_1+q_0I-q_3J+q_2IJ.\]
We are left to prove that 
\begin{equation}\label{eq:1}
\frac{\partial f}{\partial x_2}(y)=J\widehat{f'_s}(y)\quad \text{and}\quad 
 \frac{\partial f}{\partial x_3}(y)=IJ\widehat{f'_s}(y).
\end{equation}

Assume $y\in\OO\setminus\R$.
Let $y=\alpha+I\beta$ with $\alpha,\beta \in \R$,  $\beta>0$, and let $z:=\alpha+i\beta \in \cc$. The smooth arc $\gamma_J:\R \rightarrow \s_y$ defined by
$\gamma_J(t):=\alpha+I\beta\cos\left(t/\beta\right)+J\beta\sin\left(t/\beta\right)$
has tangent vector $\gamma'_J(0)=J$ at $y$. Moreover, $f(\gamma_J(t))=F_1(z)+\left(\frac{\gamma_J(t)-\alpha}{\beta}\right)F_2(z)$. Therefore
\begin{align}\notag
f(\gamma_J(t))-f(y)&=\left(F_1(z)+\left(\frac{\gamma_J(t)-\alpha}{\beta}\right)F_2(z)\right)-\left(F_1(z)+\left(\frac{\gamma_J(0)-\alpha}{\beta}\right)F_2(z)\right)\\\notag
&=\left(\gamma_J(t)-\gamma_J(0)\right)f'_s(y)
\end{align}
and then 
\[
\frac{\partial f}{\partial x_2}(y)=\lim_{t \to 0}\frac{f(\gamma_J(t))-f(y)}{t}=\gamma'_J(0)f'_s(y)=Jf'_s(y).
\]
The second equality in \eqref{eq:1} is proved in the same way using the analogous curve $\gamma_{IJ}$.

If $y\in\OO\cap\R$, the result follows by passing to the limit as $x\in\left(\OO\setminus\R\right)\cap\cc_I$ tends to $y$.
\end{proof}

For any $I\in\s_\hh$, let $\pi_I:\hh\rightarrow\hh$ denote the orthogonal projection onto the real vector subspace $\cc_I$ and let $\pi_I^\bot=\mathit{id}_	\HH-\pi_I$. Given a quaternion $x$, let $L_x$ and $R_x$ be respectively the operators of left and right multiplication by $x$.

Proposition~\ref{pro:matrix} allows to obtain a property of the differential $df_y$ already observed in \cite[\S3]{TwistorJEMS}.

\begin{corollary}\label{cor:df}
Let $f\in\sr(\OO)$. If $y\in\OO\cap\cc_I$, then the differential of $f$ at $y$ can be written as
\[df_y=R_{\cd(y)}\circ\pi_I+R_{\widehat{f'_s}(y)}\circ\pi_I^\bot.
\]
In particular, if $y\in\OO\cap\R$, then $df_y=R_{\cd(y)}$.
\end{corollary}
\begin{proof}
Let $J$ be as in Proposition~\ref{pro:matrix}. Since $\pi_I(\hh)=\cc_I=\mr{Span}(1,I)$ and $\pi_I^\bot(\hh)=\mr{Span}(J,IJ)$, the result comes easily from the form of the representing matrix $J_f(y)$.
\end{proof}

We can now generalize a result proved in \cite[Theorem~1]{indam2013}. For any $I\in\s_{\hh}$, let $(\hh,L_I)$ denote the complex manifold obtained equipping $\hh$ with the complex structure $L_I$.

\begin{corollary}\label{cor:complex}
Let $f\in\sr(\OO)$ and let $y\in\OO\cap\cc_I$. 
The differential  $df_y$ is a linear holomorphic mapping from the space $(\hh,L_I)$ into itself.
\end{corollary}
\begin{proof}
The thesis follows immediately from Corollary~\ref{cor:df} using the commutativity of $L_I$ with the operators of right multiplication and with $\pi_I$.
\end{proof}

\begin{theorem}\label{thm:signJacobian}
Every slice regular function $f\in\sr(\OO)$ preserves the orientation of $\hh\simeq\R^4$. More precisely, for each $y\in \OO$, the Jacobian matrix $J_f(y)$ has even rank and its determinant is non-negative. When $y\in\OO\cap\R$ the rank can assume only the values 0 or 4.
\end{theorem}
\begin{proof}
Let $y\in\OO$ and let $I,J\in\s_\hh$ be such that $y\in\cc_I$ and $J$ is orthogonal to $I$. 
The complex manifold $(\hh,L_I)$ is $\cc_I$-biholomorphic to $\cc_I^2$ by means of the mapping sending $x=x_0+x_1I+x_2J+x_3IJ$ to $(x_0+x_1I, x_2+x_3I)$. From Corollary~\ref{cor:complex}, the linear map $df_y:\cc_I^2\rightarrow\cc_I^2$ is holomorphic, therefore its determinant is the squared norm of a complex Jacobian determinant.
To compute the rank of $df_y$, consider the representing matrix $J_f(y)$ given in \eqref{eq:matrix}. Since the first two columns and the last two ones generate two $\cc_I$-complex subspaces, the rank is $0$, $2$ or $4$. 
The last statement follows from Corollary~\ref{cor:df}.
\end{proof}

\begin{remark}
Theorem~\ref{thm:signJacobian} implies in particular that every polynomial with quaternionic coef\-ficients on one side has non-negative Jacobian, a fact recently proved in \cite{Sakkalis} with completely different techniques. The case of quaternionic powers $x^n$ was already considered in \cite{Mawhin2005}.
\end{remark}

\begin{remark}
The evenness of the rank of $df_y$ was already proved in \cite[Proposition~3.3]{TwistorJEMS}.
\end{remark}


\section{A formula for the Jacobian}\label{sec:formulaJacobian}

The argument used in the proof of Theorem~\ref{thm:signJacobian} can be applied also to obtain an explicit formula for the Jacobian determinant of a slice regular function $f$. Equip the manifold $\hh\setminus\R$ with the quaternionic valued $\R$-bilinear form defined as follows. For $y\in\left(\hh\setminus\R\right)\cap\cc_I$, given two quaternions $u,v$ in the tangent space $T_y\left(\hh\setminus\R\right)\simeq\hh$, we set 
\[(u,v)_y:=\pi_I(u\overline v)=\langle u, v\rangle_I\in \cc_I,
\] 
where $\langle u,v \rangle_I$ denotes the standard Hermitian product on $\hh$ w.r.t.\ the complex structure $L_I$. The choice of a unit $J$ orthogonal to $I$ induces the $\cc_I$-linear isomorphism $(\hh,L_I)\simeq\cc_I^2$ sending $u=u_0+u_1I+u_2J+u_3IJ$ to $(u^1,u^2)=(u_0+u_1I,u_2+u_3I)$. The product $\langle u,v\rangle_I$ does not depend on $J$, since
\[u\overline v=(u^1+u^2J)\overline{(v^1+v^2J)}=u^1\overline{v^1}+u^2\overline{v^2}+(u^2v^1-u^1v^2)J\]
and then $\langle u,v\rangle _I=u^1\overline{v^1}+u^2\overline{v^2}=\pi_I(u\overline v)$.

\begin{theorem}\label{thm:Jacobian}
Let $f\in\sr(\OO)$ and let $y\in\OO$. If $y\in\OO\setminus\R$ then
\[\det(J_f(y))=\left|\left(\cd(y),f'_s(y)\right)_y\right|^2.
\]
If $y\in\OO\cap\R$, then $\det(J_f(y))=\left|\cd(y)\right|^4$.
\end{theorem}
\begin{proof}
Let assume $y\in\left(\OO\setminus\R\right)\cap\cc_I$. Let $\cd(y)=q=q_0+q_1I+q_2J+q_3IJ$ and $Jf'_s(y)=p=p_0+p_1I+p_2J+p_3IJ$ be as in Proposition~\ref{pro:matrix}. Set $q^1=q_0+Iq_1$, $q^2=q_2+Iq_3$, $p^1=p_0+Ip_1$, $p^2=p_2+Ip_3\in\cc_I$.
From Proposition~\ref{pro:matrix}, it follows that the differential $df_y$  has $\cc_I$-valued representing matrix
\[
J_f^{\cc_I}(y)=
\left[
\begin{array}{rrrr}
q^1&p^1\\
q^2&p^2\end{array}\right]
\]
w.r.t.\ the $\cc_I$-basis $\{1,J\}$ of  $(\hh,L_I)\simeq\cc_I^2$. Therefore the $\cc_I$-valued determinant of $df_y$ is 
\[\det(J_f^{\cc_I}(y))=q^1p^2-q^2p^1=\langle q^1+q^2J,\overline{p^2}-\overline{p^1}J\rangle _I=\langle q,-Jp\rangle _I=\left(\textstyle\cd(y),f'_s(y)\right)_y.
\]
From Corollary~\ref{cor:complex}, we obtain that 
\[\det(J_f(y))=\left|\det(J_f^{\cc_I}(y))\right|^2=\left|\left(\cd(y),f'_s(y)\right)_y\right|^2.
\]

If $y\in\OO\cap\R$, the equality $\det(J_f(y))=\big|\cd(y)\big|^4$ follows by passing to the limit as $x\in\left(\OO\setminus\R\right)\cap\cc_I$ tends to $y$.
\end{proof}

\begin{remark}\label{rem:realformula}
At a point $y\in\left(\OO\setminus\R\right)\cap\cc_I$, the Hermitian product $\langle u,v\rangle_I$ decomposes as follows:
\[\langle u,v\rangle_I=\langle u,v\rangle-I\,\omega_I(u,v),
\]
where $\langle u,v\rangle$ is the Euclidean scalar product of $u,v$ as vectors in $\R^4$, and $\omega_I(u,v)=\langle Iu,v\rangle=-\langle u,Iv\rangle=u_0v_1-u_1v_0+u_2v_3-u_3v_2$ is the corresponding fundamental form. Therefore
\[\det(J_f^{\cc_I}(y))=\left\langle\cd(y),f'_s(y)\right\rangle_I=\left\langle\cd(y),f'_s(y)\right\rangle+I\left\langle\cd(y),I f'_s(y)\right\rangle.
\]
Since $I=\im(y)/|\im(y)|$, we can write the real Jacobian in terms of the Euclidean product as
\begin{equation}\label{eq:realformula}
\det(J_f(y))=\left\langle\cd(y),f'_s(y)\right\rangle^2+\left\langle\cd(y),\frac{\im(y)}{|\im(y)|}f'_s(y)\right\rangle^2.
\end{equation}
\end{remark}

\begin{remark}
In \cite{gori_vlacci_2019}, making use of Dieudonn\'e determinant, the authors proved the following partial result: $\det(J_f(y))=0$ if and only if $\left(\textstyle\cd(y),f'_s(y)\right)_y=0$.
\end{remark}


\section{The fibers of a slice regular function} \label{sec:fibers}

In this section we describe the fibers of a slice regular function $f\in\sr(\OO)$. Given $c\in \hh$, the fiber $f^{-1}(c)$ of $f$ over $c$ is the zero set $V(f-c)$. 
If $\OO$ is a slice domain, then the zero set of a not identically vanishing slice regular function on $\OO$ consists of isolated points or isolated 2-spheres of the form $\s_x$ (see e.g.~\cite[Theorem~3.12]{GeStoSt2013}). Therefore if $f$ is not constant, every fiber of $f$ has this structure. If $\OO$ is a product domain, a new phenomenon appears. We will show that a fiber of a not slice-constant $f$ can contain also a complex analytic curve. We will also see that the structure of the fiber is controlled by the normal function $N(f-c)$.

On product domains $\OO$ it can happen that $N(f)\equiv0$ even if $f\not\equiv0$. In \cite[Corollary~4.17]{AlgebraSliceFunctions}, it was shown that if $N(f)\not\equiv0$, then $V(f)$ is a union of isolated points or isolated $2$-spheres $\s_x$. Let $c\in f(\OO)$. If $N(f-c)\not\equiv0$, then the fiber $f^{-1}(c)$ is a union of isolated points or isolated $2$-spheres. If instead $N(f-c)\equiv0$, then every $2$-sphere $\s_x$, with $x\in\OO$, intersects the fiber $V(f-c)$ in a point, or it is entirely contained in $V(f-c)$ (see e.g.~\cite[Theorem~17]{GhPe_AIM}). In the latter case, $\s_x$ is contained in the degenerate set $D_f$ of $f$, provided $x\not\in\R$.

In the next statement, the complex structure on $\s_\hh\simeq\cc\pp^1$ is the one induced by the structure $\jj$ on $\hh\setminus\rr$ defined at $x$ by left multiplication by $\im(x)/|\im(x)|$.

\begin{notation}
We define $D^+:=\{\alpha+i\beta\in D\,:\,\beta>0\}$ and $\cc^+_J:=\{\alpha+J\beta\in\cc_J \,:\,\beta>0\}$ for any $J\in\s_\HH$.
\end{notation}

\begin{proposition} \label{prop:wings}
Let $\OO=\OO_D$ be a product domain and let $f\in\sr(\OO)$ be non-constant. Fix $c\in f(\OO)$ such that $N(f-c)\equiv0$. Then there exists a holomorphic function $\phi : D^+\rightarrow\s_\hh$ such that the fiber $f^{-1}(c)$ is equal to $D_f \cup W_{f,c}$ where
\[W_{f,c}=\{\alpha+\phi(\alpha,\beta)\beta\in\OO\,:\,\alpha+i\beta\in{D^+}\}.
\]
Moreover, $D_f$ is empty or is a union of isolated spheres, and the map sending $z=\alpha+i\beta\in{D^+}$ to $\alpha+\phi(\alpha,\beta)\beta\in W_{f,c}$ is holomorphic from ${D^+}$ to $(\hh\setminus\R,\jj)$. We call the complex analytic curve $W_{f,c}$ a \emph{wing} of $f$ (of value $c$ induced by $\phi$).
\end{proposition}
\begin{proof}
Let $f=\I(F_1+\ui F_2)$. Let $\widetilde{D^+}$ be the set of points $z\in D^+$ such that the $2$-sphere $\OO_{\{z\}}$ is not contained in $V(f-c)$. Note that $D^+\setminus\widetilde{D^+}=\{z\in D^+\,|\,F_2(z)=0\}$. By \cite[Theorem~4.11]{AlgebraSliceFunctions} we know that $\cc^+_J\cap V(f-c)$ is closed and discrete in $\OO^+_J=\OO\cap\cc^+_J$ or it is the whole set $\OO^+_J$. Since $f$ is non-constant, the set $\cc^+_J\cap V(f-c)$ is discrete for at least one $J\in\s_\hh$. It follows that $D^+\setminus\widetilde{D^+}$ is a closed and discrete subset of $D^+$. Let $z=\alpha+i\beta\in\widetilde{D^+}$ and let $\s_x=\OO_{\{z\}}$. Bearing in mind that $F_2(z)\ne0$, we deduce that the (unique) point $\alpha+\phi_z\beta$ in the intersection $V(f-c)\cap \s_x$ is given by the formula
\[
\phi_z =(c-F_1(z))F_2(z)^{-1}.
\]
This formula defines a real analytic map $\phi : \widetilde{D^+}\rightarrow\s_\hh$, sending $z$ to $\phi_z$. Deriving the equality
\[F_1(\alpha,\beta)+\phi(\alpha,\beta)F_2(\alpha,\beta)=c,\]
and using the holomorphicity of $F_1+\ui F_2$, we get 
\[\dd\phi\alpha F_2=-\phi\dd\phi\beta F_2,\quad \dd\phi\beta F_2=\phi\dd\phi\alpha F_2.
\]
Since $F_2\ne0$ on $\widetilde{D^+}$, we deduce that the map $\phi:\widetilde{D^+}\rightarrow(\s_\hh,\jj)$ is holomorphic. It remains to show that $\phi$ extends holomorphically to $D^+$.
Let $z_0\in D^+\setminus\widetilde{D^+}$ and assume that the punctured open disc $\dot B(z_0,r)=B(z_0,r)\setminus\{z_0\}$ of $\cc$ is contained in $\widetilde{D^+}$. Let $J\in\s_\hh$ be fixed. As we have just recalled $\cc^+_J\cap V(f-c)$ is closed and discrete or it is the whole set $\OO^+_J$. In the latter case, $\phi$ is constantly equal to $J$ on $\widetilde{D^+}$, and then it extends to $D^+$. In the other case, taking a smaller $r>0$ we can assume that $\OO_{\dot B(z_0,r)}$ does not intersect $\cc^+_J\cap V(f-c)$. This means that $J\not\in\phi(\dot B(z_0,r))$. Repeating the argument with other elements of $\s_\hh$, we obtain $r_0>0$ such that $\phi(\dot B(z_0,r_0))$ avoids at least three points in $\s_\hh$. The Big Picard Theorem permits to conclude. If $\psi:D^+ \to (\hh\setminus\R,\jj)$ denotes the function $\psi(\alpha+i\beta)=\alpha+\phi(\alpha,\beta)\beta$, then the equality $\dd\phi\beta=\phi\dd\phi\alpha$ implies at once  $\dd\psi\beta=\phi\dd\psi\alpha$.
\end{proof}

Note that $W_{f,c}$ is closed in $\OO$ and it is a real analytic submanifold of $\OO$ of dimension $2$.

\begin{corollary}\label{cor:fibers}
Let $f\in\sr(\OO)$ be non-constant.
If $\OO$ is a slice domain, then every fiber of $f$ is the union of a set of isolated points and a set of isolated 2-spheres of the form $\s_x$. If $\OO$ is a product domain, every fiber of $f$ is the union of a set of isolated points, a set of isolated 2-spheres of the form $\s_x$ and (possibly) one wing $W_{f,c}$. Moreover $f^{-1}(c)\supset W_{f,c}$ if and only if $N(f-c)\equiv0$. In the latter case $f^{-1}(c)=D_f \cup W_{f,c}$.
If there are at least two fibers containing a wing, then $D_f=\emptyset$.
\end{corollary}
\begin {proof}
It remains to prove the last statement. Since two distinct fibers cannot intersect a $2$-sphere $\s_x$ where $f$ is constant, if there are two wings the degenerate set $D_f$ is empty.
\end{proof}

\begin{notation}
Let $f\in\sr(\OO)$. If $\OO$ is a product domain, we denote by $W_f$ the union of all the wings of $f$. If $\OO$ is a slice domain, we say that $f$ has no wing and we define $W_f:=\emptyset$.
\end{notation}

In the following we shall prove some results about the family of wings that a slice regular function can have. In particular, we will show that a not slice-constant regular function has no wings if it is slice-preserving. More generally, this holds for slice regular functions that are slice-preserving up to an invertible quaternionic affine transformation, i.e.\ in the set
\[
\widetilde{\sr}_\rr(\OO)=\{f\in\sr(\OO) \,:\, \exists a,b\in\hh,  \text{ $\exists g\in\sr_{\rr}(\OO)$ such that $a\neq0$, $f=ga+b$}\}.
\]
We shall consider also the larger set of slice regular functions that are one-slice-preserving up to an invertible quaternionic affine transformation:
\[
\widetilde{\sr}_\cc(\OO)=\{f\in\sr(\OO) \,:\, \exists a,b\in\hh,  \text{ $\exists g\in\sr_\cc(\OO)$ such that $a\neq0$, $f=ga+b$}\}.
\]

Note that the set of invertible quaternionic affine transformation is the group $\mathit{Aut}(\hh)$ of \emph{biregular automorphisms} of $\hh$, i.e.\ the slice regular functions $f:\hh\to\hh$ having slice regular inverse (see \cite[Theorem 9.4]{GeStoSt2013}).

Denote by $\SC(\OO)$ the set of all slice-constant functions.  Evidently, $\SC(\OO)$ coincides with the set of all constant functions on $\OO$ if and only if $\OO$ is a slice domain. Also in the case of a product domain $\OO$, every $f\in\SC(\OO)$ belongs to $\widetilde{\sr}_\rr(\OO)$. Indeed, $f=\I(F_1+\ui F_2)$ with $F_1$ and $F_2$ constant; hence we get $f(x)=g(x)a+b$, where $a=F_2$, $b=F_1$ and $g(x)=\frac{\im(x)}{|\im(x)|}\in\sr_\rr(\OO)$.

For every slice and product domains $\OO=\OO_D$, we have the following chain of inclusions:
\[
\SC(\OO)\subsetneq\widetilde{\sr}_\rr(\OO)\subsetneq\widetilde{\sr}_\cc(\OO)\subsetneq\sr(\OO).
\]
In addition, thanks to the representation formula, if $\OO$ is a product domain and $f\in\SC(\OO)$ is non-constant then the set $\OO\cap\cc_J^+$ is a fiber of $f$ for each $J\in\s_\hh$.

\begin{proposition}\label{pro:realfibers}
Every $f\in\widetilde{\sr}_\rr(\OO)\setminus\SC(\OO)$ has no wings.
\end{proposition}
\begin{proof}
It is sufficient to prove the result for $f\in\sr_\rr(\OO)\setminus\SC(\OO)$. Indeed, if $f\in\widetilde{\sr}_\rr(\OO)\setminus\SC(\OO)$, then  $f=ga+b$, with $g\in\sr_{\rr}(\OO)\setminus\SC(\OO)$, $a,b\in\hh$, $a\ne0$, and the fiber $f^{-1}(c)$ coincides with the fiber $g^{-1}((c-b)a^{-1})$ of $g$. We then assume that $f\in\sr_{\rr}(\OO)\setminus\SC(\OO)$. As seen above, we can suppose that $\OO$ is a product domain. Let $f=\I(F)$, $F=F_1+\ui F_2$, with $F_1$, $F_2$ real-valued. Suppose that there is a fiber $f^{-1}(c)$ which contains a wing $W_{f,c}$. This implies that $N(f-c)\equiv0$. In particular,  
\[
(F_1-\re(c))F_2=\langle F_1-c,F_2\rangle=0
\]
and then $F_1\equiv\re(c)$ would be constant and hence $f$ would be slice-constant.
\end{proof}

\begin{proposition}\label{pro:one-slice-fibers}
Let $\OO=\OO_D$ be a product domain and let $f\in\widetilde{\sr}_\cc(\OO)\setminus\widetilde{\sr}_\rr(\OO)$. Suppose there are at least two fibers of $f$ containing (and then equal to) a wing. Then $D_f=\emptyset$ and $f$ has infinite wings, parametrized by a circle $C$. More precisely, if $f=ga+b$, with $g\in\sr_{\cc_J}(\OO)\setminus\widetilde{\sr}_\rr(\OO)$ for $J\in\s_\hh$, $a,b\in\hh$, $a\ne0$, and $f^{-1}(d)=W_{f,d}$, then $f^{-1}(c)=W_{f,c}$ if and only if the quaternion $c$ belongs to the circle $C$ defined by the following two conditions:
\begin{equation}\label{eq:C}
|c-b|=|d-b| \; \text{ and } \; (c-d)a^{-1}\in\cc_J^\bot.
\end{equation}
Furthermore, the set $W_f$ coincides with $f^{-1}(C)$, it is a real analytic submanifold of $\OO$ and the restriction $f|:W_f\to C$ is a trivial fiber bundle with fiber $D^+$. More precisely, the map $\chi:D^+\times C\to W_f$, defined by $\chi(z,c):=\alpha+(c-F_1(z))F_2(z)^{-1}\beta$ for every $z=\alpha+i\beta\in D^+$ and $c\in C$, is a real analytic isomorphism such that $(f|\circ\chi)(z,c)=c$ for every $(z,c)\in D^+\times C$.
\end{proposition}
\begin{proof}
It is sufficient to prove the result for $f\in\sr_{\cc_J}(\OO)\setminus\widetilde{\sr}_\rr(\OO)$. First observe that $D_f=\emptyset$. Otherwise $f$ would be constant, contradicting the hypothesis $f\not\in\widetilde{\sr}_\rr(\OO)$. If $f^{-1}(d)\supset W_{f,d}$, then $N(f-d)\equiv0$, i.e.\  $|F_1-d|=|F_2|$ and $\langle F_1-d,F_2\rangle=0$. Let $c\in\hh$ such that $|c|=|d|$ and $c-d\in\cc_J^\bot$. Then $\langle F_1-c,F_2\rangle=\langle F_1-d,F_2\rangle=0$ and
\[|F_1-c|^2=|F_1|^2+|c|^2-2\re\langle F_1,c\rangle=|F_1|^2+|d|^2-2\re\langle F_1,d\rangle=|F_1-d|^2=|F_2|^2,
\]
that is $N(f-c)\equiv0$. Conversely, if $N(f-d)=N(f-c)\equiv0$, then $\langle F_1-c,F_2\rangle=0=\langle F_1-d,F_2\rangle$ and
$|F_1-d|=|F_2|=|F_1-c|$. It follows that 
\begin{equation}\label{eq:cd}
\langle d-c,F_2\rangle=0,\quad  |c|^2-|d|^2=2\re\langle F_1,c-d\rangle.
\end{equation}

Let $w\in D^+$ and let $a:=F_2(w)\ne0$. If $F_2a^{-1}$ were real-valued, then the holomorphy of $F$ would imply that $fa^{-1}$ is, up to an additive constant, a slice-preserving function, i.e.\ $f\in\widetilde{\sr}_\rr(\OO)$, which is a contradiction. We can then assume that the real vector subspace $\langle F_2(D)\rangle$ of $\hh$ generated by the image of $F_2$ is the plane $\cc_J$. By \eqref{eq:cd}, $N(f-c)\equiv0$ if and only if $c$ satisfies \eqref{eq:C} (with $a=1$ and $b=0$). If $f$ has at least two wings, then $d\ne0$ and the set $C$ defined by \eqref{eq:C} is a circle.

Thanks to Proposition \ref{prop:wings}, we know that $f^{-1}(c)=W_{f,c}$ for each $c\in C$. Consequently, $f^{-1}(C)=W_f$. To complete the proof, it is now sufficient to observe that the real analytic map $W_f\to D^+\times C$, $y\mapsto(\re(y)+i|\im(y)|,f(y))$ is the inverse of $\chi$.
\end{proof}

\begin{proposition}\label{prop:fibers}
Let $\OO=\OO_D$ be a product domain and let $f\in\sr(\OO)\setminus\widetilde{\sr}_\cc(\OO)$. Then there exist at most two fibers of $f$ containing a wing.
\end{proposition}
\begin{proof}
Let $f=\I(F)\in\sr(\OO)\setminus\widetilde{\sr}_\cc(\OO)$, $F=F_1+\ui F_2$. Suppose that there is at least one fiber $f^{-1}(d)$ which contains a wing $W_{f,d}$. We can suppose that $d=0$, otherwise we consider the slice regular function $f-d$. This means that  
$N(f)\equiv0$, i.e.\ 
\begin{equation}\label{eq:nf}
|F_1|=|F_2|,\quad \langle F_1,F_2\rangle=0.
\end{equation} 
Let $c\ne0$. The fiber $f^{-1}(c)$ contains a wing $W_{f,c}$ if and only if $N(f-c)\equiv0$, that is $|F_1-c|=|F_2|$ and $\langle F_1-c,F_2\rangle=0$. By \eqref{eq:nf}, the latter equations are equivalent to the following
\begin{equation}\label{eq:c}
2\langle F_1,c\rangle=|c|^2,\quad \langle F_2,c\rangle=0.
\end{equation}
Let $w\in D^+$ be such that $F_2(w)\ne0$. We can suppose that $F_2(w)=1$, otherwise we replace $f$ with $fF_2(w)^{-1}$. We distinguish three cases. First suppose that $F_2$ is real-valued. In this case, the holomorphy of $F$ would imply that $f$ is, up to an additive constant, a slice-preserving function, which is a contradiction. Assume now that the real vector subspace $\langle F_2(D)\rangle$ of $\hh$ generated by the image of $F_2$ is a $\cc_J$-plane for some $J\in\s_\hh$. Then, using again the holomorphy of $F$, we infer the existence of a constant $q\in\hh$ such that $F_1-q$ is $\cc_J$-valued. Therefore $f\in\widetilde{\sr}_\cc(\OO)$ which is a contradiction. 
The third case is the one in which the image $F_2(D)$ contains three elements $\{1,q,q'\}$ independent over $\rr$.
From the second equation in \eqref{eq:c} we get that $c$ belongs to a real line of $\HH\simeq\rr^4$ through the origin. Being $F_1$ not identically zero, from the first equation in \eqref{eq:c} we have that $c$ belongs to a sphere through the origin. Therefore there is at most one value $c\ne0$ satisfying~\eqref{eq:c}. 
\end{proof}

Combining the results \cite[Proposition 3.9 \& Theorem 3.12]{GeStoSt2013} and \cite[Corollary 4.17]{AlgebraSliceFunctions} mentioned above with Propositions \ref{prop:wings}, \ref{pro:realfibers} and \ref{pro:one-slice-fibers}, one immediately obtains a quite explicit description of all the fibers of an arbitrary slice regular function.

\begin{theorem} \label{thm:fibersformula}
Let $f=\I(F_1+\ui F_2)\in\sr(\OO_D)$ and let $c\in\hh$. Define $D_{\sss\geq}$, $\n(f,c)$ and $\n'_s(f,c)$ by
\begin{align*}
D_{\sss\geq}&:=\{\alpha+i\beta\in D\,:\,\beta\geq0\},\\
\n(f,c)&:=\{z \in D_{\sss\geq} \,:\, |F_1(z)-c|=|F_2(z)|, \,\langle F_1(z)-c,F_2(z)\rangle=0\},\\
\n'_s(f,c)&:=\{z \in\n(f,c) \,:\, F_2(z)=0\}.
\end{align*}
Then each of the subsets $\n(f,c)$ and $\n'_s(f,c)$ of $D_{\sss\geq}$ is closed and discrete or it coincides with the whole $D_{\sss\geq}$, and one of the following holds:
\begin{enumerate}
 \item If $\n(f,c)=\n'_s(f,c)=D_{\sss\geq}$ then $f\equiv c$.
 \item If $\n(f,c)=D_{\sss\geq}$ and $\n'_s(f,c)$ is discrete then $\OO_D$ is a product domain and $f^{-1}(c)=D_f\cup W_{f,c}$, where $D_f$ coincides with the circularization of $\n'_s(f,c)$ and $W_{f,c}$ is a wing of $f$ induced by the unique holomorphic function $\phi:D_{\sss\geq}=D^+\to\s_\hh$ such that $\phi(z)=(c-F_1(z))F_2(z)^{-1}$ for every $z \in D^+\setminus\n'_s(f,c)$.
 \item If $\n(f,c)$ and $\n'_s(f,c)$ are discrete, then $f^{-1}(c)=\mathit{SZ}\sqcup\mathit{RZ}\sqcup\mathit{INRZ}$, where
\begin{itemize}
 \item $\mathit{SZ}=\bigcup_{z\in\n'_s(f,c)\setminus\R}\OO_{\{z\}}$ is the set of spherical zeros of $f-c$,
 \item $\mathit{RZ}=\bigcup_{z\in\n'_s(f,c)\cap\R}\{z\}$ is the set of real zeros of $f-c$,
 \item $\mathit{INRZ}=\bigcup_{z=\alpha+i\beta \in \n(f,c)\setminus\n'_s(f,c)}\left\{\alpha+(c-F_1(z))F_2(z)^{-1}\beta\right\}$ is the set of isolated non-real zeros of $f-c$.
\end{itemize}
\end{enumerate}
In particular, $f(\OO_D)=\{c\in\hh\,:\,\n(f,c)\ne\emptyset\}$.

Furthermore $W_f=\emptyset$ when $\OO_D$ is a slice domain. When $\OO=\OO_D$ is a product domain the set $W_f$ is closed in $\OO$ and it holds:
\begin{itemize}
 \item[4.] $W_f=\OO$ if $f\in\SC(\OO)$.
 \item[5.] $W_f=\emptyset$ if $f\in\widetilde{\sr}_\rr(\OO)\setminus\SC(\OO)$.
 \item[6.] If $f\in\widetilde{\sr}_\cc(\OO)\setminus\widetilde{\sr}_\rr(\OO)$, then either $W_f=\emptyset$, or $W_f$ coincides with a wing, or $W_f$ is a real analytic submanifold of $\OO$ of dimension 3. In the latter case $W_f$ is real analytic isomorphic to $D^+\times C$, where $C$ is the circle of $\hh$ defined in \eqref{eq:C}.
 \item[7.] If $f\in\sr(\OO)\setminus\widetilde{\sr}_\cc(\OO)$, then either $W_f=\emptyset$, or $W_f$ coincides with a wing or with the union of two disjoint wings.
\end{itemize}
\end{theorem}

All the six possibilities mentioned in points {\it 6} and {\it 7} of the preceding statement can happen.

\begin{examples}\label{ex:wings}
Let $\OO:=\hh\setminus\rr$ and let $\eta \in\SC(\OO)$ be the function $\eta(x)=\frac{1}{2}(1-I_xi)$, where $I_x:=\frac{\im(x)}{|\im(x)|}$. Define $f_1,f_2,f_3 \in\sr_{\cc_i}(\OO)\setminus\widetilde{\sr}_\rr(\OO)$ as follows:
\[
f_1(x):=x^2-2xi, \qquad f_2(x):=x\eta(x), \qquad{f_3(x):={\textstyle\big(x+\frac{1}{x}\big)\eta(x)-\frac{1}{x}}}.
\]

Making use of Theorem \ref{thm:fibersformula} it is easy to describe the fibers of the preceding functions over an arbitrary quaternion $c=c_0+c_1i+c_2j+c_3k$ with $c_0,c_1,c_2,c_3\in\R$:
\begin{itemize}
\item[1.] All the fibers of $f_1$ contains at most two points. It follows that $W_{f_1}=\emptyset$.
\item[2.] \label{ex:1wings}
$f_2$ has one planar wing $f_2^{-1}(0)=W_{f_2,0}=\cc^+_{-i}$. All the other fibers of $f_2$ contains at most one point. If $c\ne0$, we have: $f_2^{-1}(c)=\emptyset$ if and only if $c_1\le0$, and $f_2^{-1}(c)$ is a singleton if and only if $c_1>0$. It follows that $W_{f_2}=\cc^+_{-i}$. Moreover, $f_2(\OO)=\{c_1>0\}\cup\{0\}$.
 \item[3.] $f_3$ coincides with the identity on $\s_\hh\cup\cc_i^+$ and has normal function $N(f)\equiv-1$. It holds $N(f_3-c)\equiv0$ if and only if $c=aj+bk$ for $a,b\in\rr$, $a^2+b^2=1$. Therefore $f_3$ has a circle of wings as fibers. Consequently, $W_{f_3}$ is a real analytic submanifold of $\OO$ of dimension 3.
\end{itemize}

Consider now the functions $f_4,f_5,f_6 \in\sr(\OO)\setminus\widetilde{\sr}_\cc(\OO)$ defined by
\[
f_4(x):=x^3+x^2i+xj, \qquad f_5(x):=(x^2+xj)\cdot\eta(x), \qquad f_6(x):=x^{-1}f_5(x)=(x+j)\cdot \eta(x).
\]
Also in this case it is easy to verify the following:
\begin{itemize}
 \item[4.] 
$f_4$ has no wings. Consequently, $W_{f_4}=\emptyset$. 
 \item[5.] $f_5$ has only one non-planar wing $f_5^{-1}(0)=W_{f_5,0}=W_{f_5}$ given by
\[
W_{f_5,0}:=\{\alpha+\phi_5(\alpha+i\beta)\beta\in\OO \,:\, \alpha+i\beta\in\cc^+\},
\]
where $\phi_5:\cc^+\to\s_\hh$ is defined as follows
\[
\phi_5(\alpha+i\beta):=\frac{(1-\alpha^2-\beta^2)i-2\beta j+2\alpha k}{1+\alpha^2+\beta^2}.
\]
 \item[6.]\label{ex:2wings}
 $f_6$ has exactly two wings (see \cite[Example 2]{AltavillaAdvGeo}): the planar wing $f_6^{-1}(j)=W_{f_6,j}$ equal to $\cc^+_{-i}$ and the non-planar wing $f_6^{-1}(0)=W_{f_6,0}$ equal to the non-planar wing $W_{f_5,0}$ of $f_5$ (see Remark \ref{rem:wing-selection} below). It follows that  $W_{f_6}=\cc^+_{-i}\cup W_{f_5,0}$.
\end{itemize}

We conclude with two examples illustrating the case covered by Proposition \ref{prop:wings} in which a fiber of a slice regular function is equal to the union of its degenerate set and a wing. Define $f_2^*\in\sr_{\cc_i}(\OO)\setminus\widetilde{\sr}_\rr(\OO)$ and $f_5^*\in\sr(\OO)\setminus\widetilde{\sr}_\cc(\OO)$ by $f_2^*(x):=(x^2+1)\eta(x)$ and $f_5^*(x):=(x^2+1)f_5(x)$. It holds:
\begin{itemize}
 \item[$\mathit{2^*}$.] $f_2^*$ has a unique wing $W_{f_2^*,0}=\cc^+_{-i}=W_{f_2^*}$ and $(f_2^*)^{-1}(0)=\s_\hh\cup\cc^+_{-i}$.
 \item[$\mathit{5^*}$.] $f_5^*$ has a unique wing $W_{f_5^*,0}=W_{f_5,0}=W_{f_5^*}$ and $(f_5^*)^{-1}(0)=\s_\hh\cup W_{f_5,0}$.
\end{itemize}
\end{examples}

\begin{remark}
If $f$ is not slice-constant and it has at least two wings as fibers, then at most one of them can be a half-plane $\cc^+_J$. If not, the representation formula would imply that $f$ is slice-constant.
\end{remark}

The next result is a criterion for the existence of at most one wing for a slice regular function. \emph{From now on, given any subset $S$ of $\cc$, we denote by $\mr{cl}(S)$ the Euclidean closure of $S$ in $\cc$.}

\begin{lemma}\label{lem:one-wing}
Let $\OO=\OO_D$ be a product domain and let $f=\I(F_1+\ui F_2)\in\sr(\OO)$. Suppose there exists a point $z'\in\mr{cl}(D^+)$ such that $\lim_{D^+\ni z\to z'}F_2(z)=0$. Then $f$ has at most one wing.
\end{lemma}
\begin{proof}
Suppose $f$ has two distinct wings $W_{f,c}$ and $W_{f,d}$, where $c$ and $d$ are two different quaternions. Let $\phi_c$ and $\phi_d$ be the holomorphic maps from $D^+$ to $\s_\hh$ inducing $W_{f,c}$ and $W_{f,d}$, respe\-ctively. Since $F_1+\phi_cF_2=c$ and $F_1+\phi_dF_2=d$ on $D^+$, it follows that $(\phi_c-\phi_d)F_2=c-d$. Bearing in mind that $|\phi_c(z)-\phi_d(z)|\leq|\phi_c(z)|+|\phi_d(z)|=2$ for every $z\in D^+$, we deduce
\[
c-d=\lim_{D^+\ni z\to z'}(\phi_c(z)-\phi_d(z))F_2(z)=0,
\]
which is a contradiction.
\end{proof}

As a consequence we obtain a `wing selection lemma':

\begin{lemma}\label{lem:wing-selection}
Let $\OO=\OO_D$ be a product domain such that $\mr{cl}(D^+)\cap\rr\neq\emptyset$ and let $f\in\sr(\OO)$ be a slice regular function having at least one wing $W_{f,c}$. Suppose there exists a point $r\in\mr{cl}(D^+)\cap\rr$ and a neighborhood $U$ of $r$ in $\hh$ such that $|f|$ is bounded on $U\cap\OO$. Then the slice regular function $g\in\sr(\OO)$ defined by $g(x):=c+(x-r)(f(x)-c)$ has a unique wing $W_{g,c}$ and it holds $g^{-1}(c)=W_{g,c}=W_{f,c}$.
\end{lemma}
\begin{proof}
First, observe that $g^{-1}(c)=V(g-c)=V(f-c)=W_{f,c}$. Let $f=\I(F_1+\ui F_2)$ and $g=\I(G_1+\ui G_2)$. Since $|f|$ is bounded locally at $r$ in $\OO$, $|F_1|$ and $|F_2|$ are bounded locally at $r$ in $D^+$. Observe that $G_2(z)=(\alpha-r)F_2(z)+\beta(F_1(z)-c)$ for every $z=\alpha+i\beta\in D^+$. Consequently, $\lim_{D^+\ni z\to r}G_2(z)=0$. The preceding lemma implies the statement.
\end{proof}

\begin{remark}\label{rem:wing-selection}
Lemma \ref{lem:wing-selection} applies to the functions $f=f_6\in\sr(\hh\setminus\rr)$ and $g=f_5\in\sr(\hh\setminus\rr)$ defined in Examples \ref{ex:wings}. Indeed, if we put $c=r=0$ in the statement of the mentioned lemma, we obtain that $f_5$ has a unique wing $f_5^{-1}(0)=W_{f_5,0}=W_{f_6,0}$, as asserted in Examples \ref{ex:wings}. Similar considerations can be repeated if $f=\eta$ and $g=f_2$.
\end{remark}

\begin{remark}
In the statement of Lemma \ref{lem:one-wing}, the hypothesis `$\,\lim_{D^+\ni z\to z'}F_2(z)=0$' can be weakened by requiring the existence of a sequence $\{z_n\}_n$ in $D^+$ converging to $z'$ such that the sequence $\{F_2(z_n)\}_n$ converges to $0$. As an immediate application of this stronger version, we have the following: if $f=\I(F_1+\ui F_2)\in\sr(\OO_D)$ has at least two wings then $\inf_{\mr{cl}(D^+)}|F_2|>0$.   
\end{remark}

We conclude this section describing a technique to  construct slice regular functions $f$ with tridimensional $W_f$.

\begin{proposition}\label{prop:schwarz}
Let $\OO=\OO_D$ be a product domain and let $g:\OO\cap\cc_i^+\to\cc_i$ be a holomorphic function such that $g$ is non-constant and nowhere zero. Denote by $f\in\sr(\OO)$ the unique slice regular function such that
\[
\text{$f(x)=g(x)\,$ for each $\,x\in \OO\cap\cc^+_{i}\;\;\;$ and $\;\;\;\displaystyle f(x)=-\frac{1}{\,\overline{g(\overline{x})}\,}\,$ for each $\,x\in\OO\cap\cc_{-i}^+$}.
\]
Then the fiber $f^{-1}(c)$ is a wing if and only if $c\in\cc_i^\perp$ and $|c|=1$. Consequently, $W_f$ is a real analytic submanifold of $\OO$ of dimension 3. 
\end{proposition}
\begin{proof}
Denote by $F=F_1+\ui F_2:D^+\to\hh\otimes\cc$ the stem function inducing $f$. We have:
\[\textstyle
\text{$F_1(z)=\frac{1}{2}\left(g(z)-\frac{1}{\,\overline{g(z)}\,}\right)\;$ and $\; F_2(z)=-\frac{i}{2}\left(g(z)+\frac{1}{\,\overline{g(z)}\,}\right)\;$ for every $\,z\in\cc^+=\cc^+_i$.}
\]
By a direct computation we see that $\langle F_1(z),F_2(z)\rangle=0$ and $|F_1(z)|^2-|F_2(z)|^2=-1$ for every $z\in D^+$, i.e.\ $N(f)\equiv -1$. It follows that, given $c\in\hh$, $N(f-c)\equiv0$ if and only if $\langle c,F_2(z)\rangle=0$ and $|c|^2-2\langle c,F_1(z)\rangle=1$ for every $z\in D^+$. Since $F_2$ is not constant (because $g$ is not), we deduce $c\in\cc_i^\perp$ and $|c|=1$. 
\end{proof}


\section{Singular set and quasi-openness}\label{sec:singularset}

Following the notation of \cite[\S8.5]{GeStoSt2013}, we define the \emph{singular set} of $f\in\sr(\OO)$ as the following real analytic subset $N_f$ of $\OO$:
\[
N_f:=\{x\in\OO\,:\,df_x \text{ is not invertible}\}=\{x\in\OO\,:\,\det(J_f(x))=0\}.
\]

We can apply Theorem~\ref{thm:Jacobian} to describe the singular set by means of slice and spherical derivatives.  This description is equivalent to the one given in \cite[Proposition~8.18]{GeStoSt2013}.

\begin{corollary}\label{cor:Nf}
Let $f\in\sr(\OO)$. Then
\begin{align*}
N_f&=\left\{y\in\OO \,:\,\left(\textstyle\cd(y),\widehat{f'_s}(y)\right)_y=0\right\}
=\bigcup_{I\in\s_\hh}\left\{y\in\OO\cap\cc_I\,:\,\textstyle\cd(y)\overline{\widehat{f'_s}(y)}\in\cc_I^\bot\right\}\\
&=\left\{y\in\OO\cap\R \,:\, \textstyle\cd(y)=0\right\}\cup
  \left\{y\in\OO\setminus\R \,:\,\left\langle\textstyle\cd(y),{f'_s(y)}\right\rangle=\left\langle\textstyle\cd(y),\im(y){f'_s(y)}\right\rangle=0\right\}.
\end{align*}
In particular, $N_f$ contains $V\big(\cd\big)\cup D_f$.
\qed\end{corollary}

\begin{remark}
In \cite{Harmonicity} it was proved that the spherical derivative of a slice regular function is indeed the result of a differential operation. Given the Cauchy-Riemann-Fueter operator 
\[\dcf =\dd{}{x_0}+i\dd{}{x_1}+j\dd{}{x_2}+k\dd{}{x_3},\]
for every slice regular function $f\in\sr(\OO)$ it holds $\dcf f=-2\widehat{f'_s}$ on the whole $\OO$. Therefore we have the following equivalent description of the singular set of $f$:
\[N_f=\left\{y\in\OO :\left(\textstyle\cd(y),\dcf f(y)\right)_y=0\right\}.
\]
\end{remark}

Given $f\in\sr(\OO)$, let $\tilde f:=\cd\cdot (f'_s)^c$. The function $\tilde f$ is a slice function on $\OO\setminus\R$, induced by the stem function $\tilde F:=\frac{\partial F}{\partial z}\frac{\overline F_2}{\mr{im}(z)}$. Observe that, since ${\overline F_2}/\mr{im}(z)$ takes values in $\hh$,  the slice product here coincides with the pointwise product: $\tilde f(x)=\cd(x)\overline{f'_s(x)}$ for each $x\in\OO\setminus\R$. 

Let $y=\alpha+I\beta\in N_f\setminus\R$ be fixed (with $\alpha,\beta\in\R$, $I\in\s_\hh$) and let $p={\tilde f}^0_s(y)$, $q=\beta {\tilde f}'_s(y)$. Then $\tilde f(x)=p+Jq$ for $x=\alpha+J\beta\in\s_y$. Corollary~\ref{cor:Nf} gives
\begin{align*}\label{NfSy1}
N_f\cap\s_y&= \left\{x=\alpha+J\beta\in\s_y \;:\;\left\langle p+Jq,1\right\rangle=\left\langle p+Jq, J\right\rangle=0\right\}\\
&=\left\{x=\alpha+J\beta\in\s_y \;:\;\re(p+Jq)=\re(q-Jp)=0\right\}.\end{align*}
Let $p=p_0+p_1i+p_2j+p_3k$, $q=q_0+q_1i+q_2j+q_3k$  and $J=j_1i+j_2j+j_3k$. The set $N_f\cap\s_y$ is the intersection of the $2$-sphere $\s_y$ with a real affine subspace of $\hh\simeq\rr^4$:
\begin{equation}\label{NfSy}
N_f\cap\s_y=\s_y\cap\left\{x=\alpha+J\beta\in\hh \,:\,p_0-j_1q_1-j_2q_2-j_3q_3=q_0+j_1p_1+j_2p_2+j_3p_3=0\right\}.
\end{equation}

We now use this description of the singular set to obtain some of its basic properties. 

\begin{proposition}\label{pro:NfSy}
Let $f\in\sr(\OO)$.
 Given any $y\in\OO\setminus\rr$, one of the following holds: $N_f\cap\s_y$ is empty, it is a singleton, it consists of two distinct points, it is a circle or it is the whole $\s_y$. Moreover, the latter is true, namely $\s_y\subset N_f$, if and only if $\s_y\subset D_f$ or $\s_y\subset V\big(\cd\big)$.
\end{proposition}
\begin{proof}
Let $\tilde f$, $y=\alpha+I\beta\in N_f\setminus\R$, $p$ and $q$ be as above. By \eqref{NfSy}, $N_f\cap\s_y$ is the intersection between the $2$-sphere $\s_y$ of $\R^3\simeq\alpha+\R^3$ with one of its affine subspaces. Moreover, $N_f\cap\s_y=\s_y$, i.e. $\s_y\subset N_f$, if and only if $p=q=0$ or equivalently $\tilde f|_{\s_y}\equiv0$. Since $f'_s$ is constant on $\s_y$, if $\tilde f|_{\s_y}\equiv0$ and $f'_s(y)\neq0$ then $\s_y\subset V\big(\cd\big)$. 
\end{proof}

\begin{theorem}\label{thm:Nf}
Let $f\in\sr(\OO)$. The following holds:
\begin{enumerate}
 \item $f\in\SC(\OO)$ if and only if $N_f$ has an interior point in $\OO$ or, equivalently, $N_f=\OO$.
 \item If $f\in\widetilde{\sr}_\rr(\OO)$, then $N_f=V\big(\cd\big)\cup D_f$. In particular, $N_f$ is a circular set.
 \item Suppose $f\in\sr_{\cc_{J_0}}(\OO)$ for some $J_0\in\s_\hh$.  Then $N_f\cap\cc_{J_0}=\left(V\big(\cd\big)\cup D_f\right)\cap\cc_{J_0}$  and the set 
 $N_f^*:=N_f\setminus\left(V\big(\cd\big)\cup D_f\cup\cc_{J_0}\right)$ is empty or it is a $\s^1$-fibration in the following sense: for every $y\in N_f^*$, the set $N_f\cap\s_y$ is equal to the circle $C_y$ obtained intersecting $\s_y$ with the real affine plane of $\hh\simeq\rr^4$ through $y$ and orthogonal to $\cc_{J_0}$. Moreover, $C_y\cap\left(V\big(\cd\big)\cup D_f\cup\cc_{J_0}\right)=\emptyset$. The same properties hold for any $f=ga+b\in\widetilde{\sr}_\cc(\OO)$, with $g\in\sr_{\cc_{J_0}}(\OO)$, $a,b\in\hh$, $a\ne0$.
\end{enumerate}
\end{theorem}
\begin{proof}
We begin proving {\it 1}. If $f\in\SC(\OO)$, then $N_f=V\big(\cd\big)=\OO$. Conversely, let $U$ be a non-empty open subset of $\OO$ contained in $N_f$. We can assume $U\cap\R=\emptyset$. Let $y\in U$. The intersection $U\cap\s_y$ is a non-empty open subset of $\s_y$ and hence $p=q=0$ in \eqref{NfSy}. Therefore $\tilde f=\I(\tilde F)\equiv0$ on every $2$-sphere $\s_y$ with $y\in U$. Let $D'$ be a (non-empty) open subset of $D$ such that $\OO_{D'}=\bigcup_{x\in U}\s_x$. On $D'$ the stem function $\tilde F=\frac{\partial F}{\partial z}\frac{\overline F_2}{\im(z)}$ vanishes identically. Consequently, $\frac{\partial F}{\partial z}\equiv0$ or $F_2\equiv0$ on the connected components of $D'$. Since $F$ is holomorphic, this means that $\frac{\partial F}{\partial z}\equiv0$ on $D'$ and then $\cd\equiv0$ on the connected set $\OO$, i.e.\ $f\in\SC(\OO)$.

Let us show {\it 2}. It is sufficient to prove the result for $f\in\sr_\rr(\OO)$. In this case $p$ and $q$ are real. Then $N_f\cap\s_y\ne\emptyset$ if and only if $p=q=0$ and $\s_y\subset N_f$. By Corollary \ref{cor:Nf} and Proposition~\ref{pro:NfSy}, $N_f\setminus\R=\big(V\big(\cd\big)\cup D_f\big)\setminus\R$. Combining this equality with $N_f\cap\R=V\big(\cd\big)\cap\R$ we obtain {\it 2}.

It remains to prove {\it 3}. We can assume that $J_0=i$. Let $y=\alpha+I\beta\in N_f^*$. Since $\tilde f\in\mathcal S_{\cc_{i}}(\OO\setminus\rr)$, then $p,q\in\cc_{i}$. Let $I=i_1i+i_2j+i_3k\neq\pm i$. From \eqref{NfSy} it follows that a point $x=\alpha+J\beta$ with $J=j_1i+j_2j+j_3k$ belongs to $N_f\cap\s_y$ if and only if $p_0-j_1q_1=q_0+j_1p_1=0$. Since $y\in N_f^*$, we deduce that $p_0-i_1q_1=q_0+i_1p_1=0$ and $i_1\in(-1,1)$, $\cd(y)\ne0$, $f'_s(y)\ne0$ and so $\tilde f(y)=\cd(y)f'_s(y)\ne0$. In particular $p$ and $q$ are not both null. It follows that $j_1=i_1$ is the unique solution of the equations $p_0-j_1q_1=q_0+j_1p_1=0$ for $j_1\in \R$. Therefore $N_f\cap\s_y$ is equal to the circle $C_y=\{x=\alpha+J\beta\in\s_y\,:\,j_1=i_1\}$. Note that $C_y\cap \cc_{J_0}=\emptyset$. Also $C_y\cap D_f=\emptyset$, because $C_y\subset\s_y$ and $\s_y\cap D_f=\emptyset$.

It remains to show that $C_y\cap V\big(\cd\big)=\emptyset$. Let $x=\alpha+J\beta\in C_y$ and let $z:=\alpha+i\beta\in D$. Define $\xi:=\frac{\partial F_1}{\partial\alpha}(z)\in\cc_i$ and $\eta:=\frac{\partial F_2}{\partial\alpha}(z)\in\cc_i$. Since $\xi+I\eta=\cd(y)\ne0$, it holds that either $\xi\neq0$ or $\eta\neq0$ and hence, being $J\neq\pm i$, $\cd(x)=\xi+J\eta\neq0$. This completes the proof.
\end{proof}

Our next aim is to obtain a generalization of the Open Mapping Theorem for slice regular functions (see \cite{GeSto2009} and \cite[Theorems~7.4 and 7.7]{GeStoSt2013} for slice domains and \cite[Theorem~5.1]{AltavillaCOV2015} for product domains; see also \cite{DivisionAlgebras}). Our proof of this generalization is completely new. It is based on properties of the Jacobian.

We recall that a continuous map $g:X\rightarrow Y$ between topological spaces $X$ and $Y$ is called \emph{quasi-open} if, for each point $y\in g(X)$ and for each open set $U$ in $X$ that contains a compact connected component of $g^{-1}(y)$, $y$ is in the interior of $g(U)$. Note that if $g$ is quasi-open and each of its fibers has a compact component then $g(X)$ is open in $Y$. The map $g$ is called \emph{light} if, for each $y\in Y$, the fiber $g^{-1}(y)$ is totally disconnected. If $g$ is light and quasi-open, then $g$ is open  (see e.g.\ \cite{TitusYoung}).

\emph{From now on, given any subset $S$ of $\OO$, we denote by $\Cl(S)$ and $\partial S$ the Euclidean closure of $S$ and the boundary of $S$ in $\OO$, respectively}. 

We are now in position to present our `Quasi-open Mapping Theorem'.

\begin{theorem}\label{thm:open}
Let $f\in\sr(\OO)\setminus\SC(\OO)$. The following holds:
\begin{enumerate}
 \item $f$ is quasi-open.
 \item If $\OO$ is a slice domain, then $f(\OO)$ is open in $\hh$ and the restriction $f|_{\OO\setminus \Cl(D_f)}$ is open.
 \item If $\OO$ is a product domain, then the restriction $f|_{\OO\setminus \left(D_f\cup W_f\right)}$ is open. Moreover, if $W_f=\emptyset$, then $f(\OO)$ is open in $\hh$.
\end{enumerate}
\end{theorem}
\begin{proof}
Point {\it 1} of Theorem \ref{thm:Nf} ensures that the real analytic set $N_f$ has dimension less then four. Since the Jacobian does not change sign on $\OO$ (Theorem~\ref{thm:signJacobian}), it follows from results of Titus and Young \cite{TitusYoung} that $f$ is quasi-open.

If $\OO$ is a slice domain, then the zero set of a not identically vanishing slice regular function on $\OO$ consists of isolated points or isolated 2-spheres of the form $\s_x$. It follows that the connected components of the fibers of $f$ are compact and so $f(\OO)$ is open in $\hh$. Moreover the restriction of $f$ to the open set $\OO\setminus \Cl(D_f)$ is open because it is light, being $f^{-1}(y)\setminus \Cl(D_f)=V(f-y)\setminus \Cl(D_f)$ discrete for each $y\in\hh$.

If $\OO$ is a product domain, then thanks to the description of the fibers of $f$ (Corollary~\ref{cor:fibers}) we know that the restriction of $f$ to the open set $\OO\setminus \left(D_f\cup W_f\right)$ is light, being $f^{-1}(y)\setminus \left(D_f\cup W_f\right)$ discrete for each $y\in\hh$. It follows that also such a restriction is open. Moreover, if $W_f=\emptyset$ then the connected components of the fibers of $f$ are compact (singletons or $\s_x$ indeed) and hence $f(\OO)$ is open in $\hh$.
\end{proof}

Note that, if $f$ is a non-constant function in $\SC(\OO)$, then $f(\OO)$ is a $2$-sphere of $\hh\simeq\rr^4$.

\begin{theorem}\label{thm:covering-space}
Let $f\in\sr(\OO)$ and let $y\in\OO\setminus\R$ such that $D_f=\s_y$. Then $D_f\cap\Cl(N_f\setminus D_f)\neq\emptyset$.
\end{theorem}
\begin{proof} 
Up to restricting $\OO$ around $\s_y$, we can assume $\OO$ is a product domain. Write $y=\alpha+J\beta$ with $\alpha,\beta\in\rr$, $\beta>0$ and $J\in\s_\HH$. Define $z:=\alpha+i\beta\in\cc\setminus\rr$ and $q$ as the quaternion such that $f(\s_y)=\{q\}$. Suppose the statement is false. Then $W_f=\emptyset$ and there exists a closed disc $E$ of $\cc$ centered at $z$ and contained in $\cc\setminus\rr$ such that $\OO_E\subset\OO$ and $N_f\cap\OO_E=D_f=\s_y=f^{-1}(q)\cap\OO_E$. Note that $\OO_E\setminus\s_y$ is homeomorphic to $(E\setminus\{z\})\times\s_\HH$. In particular $\OO_E\setminus\s_y$ has the same homotopy type of $\s^1\times\s^2$; consequently its fundamental group $\pi_1(\OO_E\setminus\s_y)$ is isomorphic to $\zz$. By point {\it 3} of Theorem \ref{thm:open}, $q$ is an interior point of $f(\OO_E)$. Let $U$ be the interior of $\OO_E$ in $\OO$ and let $g:\OO_E\to\hh$ be the restriction of $f$ to $\OO_E$. The set $U$ is an open neighborhood of $g^{-1}(q)=\s_y$ in $\hh$ contained in $\OO_E$ and the map $g$ is proper. It follows that there exists an open ball $B$ of $\hh$ centered at $q$ such that $B\subset g(\OO_E)$ and $g^{-1}(B)\subset U$. Denote by $V$ the open subset $g^{-1}(B)$ of $\hh$ and consider the restriction $\hat{g}:V\setminus\s_y\to B\setminus\{q\}$ of $g$. The map $\hat{g}$ is surjective and a local homeomorphism (a local diffeomorphism indeed).   

Let us prove that $\hat{g}$ is a covering space. To do this it suffices to show that the fibers of $\hat{g}$ are finite and the map $\hat{g}$ is proper (or, equivalently, the map $\hat{g}$ is closed). Suppose there exists $p\in B\setminus\{q\}$ with $\hat{g}^{-1}(p)$ infinite. Bearing in mind that $\OO_E$ is compact, there exists an accumulation point $p^*$ of $\hat{g}^{-1}(p)$ in $\OO_E$. Note that $\hat{g}^{-1}(p)\subset f^{-1}(p)$ so $p^*$ is also an accumulation point of the fiber $f^{-1}(p)$ of $f$ and $p^*\in f^{-1}(p)$. It follows that $p^*\in D_f\cup W_f$. Since $W_f=\emptyset$, we have that $p^*\in D_f=f^{-1}(q)$, which is impossible (being $p\neq q$). This proves that the fibers of $\hat{g}$ are finite. Let $C$ be a closed subset of $V\setminus\s_y$ and let $C^*$ be a closed subset of $\OO_E$ such that $C=C^*\cap(V\setminus\s_y)$. The set $C^*$ is compact in $\OO_E$ and hence the set $g(C^*)$ is closed (compact indeed) in $\hh$. Consequently, $\hat{g}(C)=g(C)=g(C^*)\cap(B\setminus\{q\})$ is closed in $B\setminus\{q\}$. This proves that the map $\hat{g}$ is proper and hence it is a covering space. 

The base space $B\setminus\{q\}$ of $\hat{g}$ is simply connected so the same is true for each connected component of its total space $V\setminus\s_y$. Choose a small loop $\gamma$ of $\OO_E\setminus\s_y$ around $y$ contained in $(V\setminus\s_y)\cap\cc^+_J$ whose homotopy class in $\OO_E\setminus\s_y$ generates $\pi_1(\OO_E\setminus\s_y)$. Since the connected component of $V\setminus\s_y$ containing the loop $\gamma$ is simply connected, the homotopy class of $\gamma$ in $V\setminus\s_y$ is trivial. Consequently the same is true for the homotopy class of $\gamma$ in $\OO_E\setminus\s_y\supset V\setminus\s_y$. This is impossible.
\end{proof}

It is known that the continuous map $g:X\to Y$ is quasi-open if and only if $\partial_Y(g(U))\subset g(\partial_X U)$ for every relatively compact open subset $U$ of $X$, where $\partial_Y(g(U))$ is the boundary of $g(U)$ in $Y$ and $\partial_X U$ the boundary of $U$ in $X$ (see \cite[Chap.~X, Theorem (4.4)]{Whyburn1958}). In particular, if $X=Y=\hh$ and $g$ is quasi-open, then $\max_{\Cl(U)}|g|=\max_{\partial U}|g|$ for every relatively compact open subset $U$ of $\hh$. Indeed, thanks to the continuity of $g$ and the compactness of $\Cl(U)$, $g(\Cl(U))\subset \Cl(g(U))\subset g(\Cl(U))$ so $g(\Cl(U))=\Cl(g(U))$ and $\partial(g(\Cl(U)))=\partial(\Cl(g(U)))\subset\partial (g(U))$. As a consequence, being $g$ quasi-open, $\partial(g(\Cl(U)))\subset g(\partial U)$; hence $\max_{\Cl(U)}|g|=\max_{\partial U}|g|$.

Thanks to the latter property of quasi-open maps we obtain the Maximum Modulus Principle for slice regular functions defined on product domains, see \cite[Theorems 7.1 and 7.2]{GeStoSt2013} for the case of slice domains and \cite[Theorems 4.2]{AltavillaCOV2015} for a partial result in the case of product domains. See also \cite{DivisionAlgebras} for a different approach.

\begin{theorem}\label{thm:MMP}
Let $\OO$ be a product domain, let $f\in\sr(\OO)$ and let $U$ be a relatively compact connected open subset of $\OO$. Then $|f|$ assumes its maximum value $M$ on $\Cl(U)$ at a point of $\partial U$. Furthermore, if $|f(p)|=M$ for some $p\in U$ and $p\in\cc_I^+$, then $f$ is constant on $\OO\cap\cc_I^+$.

\end{theorem}   
\begin{proof}
The statement is evident if $f\in\SC(\OO)$. Suppose $f$ is not in $\SC(\OO)$. By Theo\-rem \ref{thm:open}, $f$ is quasi-open. Consequently, $M=\max_{\partial U}|f|>0$. Suppose there exists $p=\xi+I\eta\in U$, $\xi,\eta\in\rr$ with $\eta>0$, $I\in\s_\hh$, such that $|f(p)|=M$. The argument exploited in the proof of \cite[Theorem 7.1]{GeStoSt2013} ensures that $f$ is constant locally at $p$ in $\cc_I$. From point {\it 3} of Theorem \ref{thm:open}, we know that $f|_{\OO\setminus(D_f\cup W_f)}$ is open. It follows that $p\in D_f$ or $p\in W_f$. 

Assume that $p\in D_f$. Then $|f(q)|=M$ for every $q \in \s_p$. 
Choose a point $q=\xi+J\eta\in\s_p\cap U$ with $J\ne I$. Let $f=\I(F_1+\ui F_2)$ and let $\OO=\OO_D$. Then $f$ is constant locally at $p$ in $\cc_I$ and locally at $q$ in $\cc_J$. Since $F_2(z)=(I-J)^{-1}(f(z_I)-f(z_J))$ and $F_1(z)=f(z_I)-IF_2(z)$ for $z=\alpha+i\beta\in D$, $z_I=\alpha+I\beta$ and $z_J=\alpha+J\beta$, it follows that $F_1$ and $F_2$ are locally constant and hence $f\in\SC(\OO)$, which is a contradiction. Therefore $p$ belongs to a wing $W_{f,c}$. Since $f$ is locally constant at $p$ in $\cc_I$, we deduce that $W_{f,c}=\OO\cap\cc_I^+$.
\end{proof}

The situation mentioned in the last assertion of the preceding statement can happen.

\begin{example}
Consider the slice regular map $f_2\in\sr(\hh\setminus\R)$ defined in Examples~\ref{ex:2wings}. Recall that $f_2^{-1}(0)=W_{f_2,0}=\cc^+_{-i}$ and $f_2(\hh\setminus\R)=\{0\}\cup\{c_0+c_1i+c_2j+c_3k\in\hh \,:\, c_1>0\}$. Define $g\in\sr(\hh\setminus\rr)$ as the (slice) reciprocal function of $i+f_2$, namely $g:=(i+f_2)^{-\bullet}$ (see \cite[\S 5.1]{GeStoSt2013} and \cite[\S 2]{AlgebraSliceFunctions} for the definition of the reciprocal function). Then $g$ has a wing $W_{g,-i}=\cc^+_{-i}$. From the pointwise formula for the reciprocal function given in \cite[Proposition~5.32]{GeStoSt2013}, it follows that $|g(x)|<1$ for every $x\in\hh\setminus\left(\R\cup W_{g,-i}\right)$ and hence $|g(y)|=1=\sup_{x\in\hh\setminus\R}|g(x)|$ for each $y\in W_{g,-i}$.
\end{example}

In the following we need a refinement of \cite[Theorem~3.9]{TwistorJEMS} 
to describe the behavior of the fibers of a slice regular function near a singular point not belonging to $\Cl(D_f)\cup W_f$. First, we recall a characterization of singular points by means of the normal function, see \cite[Proposition~3.6]{TwistorJEMS} for slice domains and \cite[Theorem~30]{AltavillaAdvGeo} for product domains.

\begin{proposition}\label{pro:Delta}
Let $f\in\sr(\OO)$ and let $y\in\OO$. Then $y\in N_f$ if and only if the total multiplicity $m_{f-f(y)}(y)$ of $f-f(y)$ at $y$ is at least two. In particular if $y\in N_f$, $N(f-f(y))\not\equiv0$ and $n$ denotes the integer $m_{f-f(y)}(y)\geq2$, then there exists $g\in\sr_\rr(\OO)$ such that $N(f-f(y))=\Delta_y^n\,g$ and $V(g)\cap\s_y=\emptyset$.
\end{proposition}

\begin{proposition}\label{pro:branch}
Let $f\in\sr(\OO)\setminus\SC(\OO)$, let $y\in N_f\setminus\left(\Cl(D_f)\cup W_f\right)$ and let $U$ be any neighborhood of $y$ in $\OO$. There exist neighborhoods $V,V'$ of $y$ in $\OO$ with $V\subset V'\subset U$, and an integer $n\ge2$ such that, for every $x\in V$, the fiber $f^{-1}(f(x))\cap V'$ of $f|_{V'}$ is finite and the sum of the total multiplicities of the points in $f^{-1}(f(x))\cap V'$ as zeros of $f-f(x)$ is equal to $n$.
\end{proposition}
\begin{proof}
Let $f=\I(F_1+\ui F_2)$. First, assume $y\in N_f\setminus\left(\Cl(D_f)\cup W_f\cup\rr\right)$. Let $y=\alpha_0+J_0\beta_0$, $z_0=\alpha_0+i\beta_0\in D^+$. For every $r\in(0,\beta_0)$, we denote by $V_r$ the open neighborhood of $y$ in $\hh$ defined by
\[
V_r:=\{\alpha+J\beta\in\hh\,:\, \alpha,\beta\in\R, J\in\s_\hh, |\alpha+i\beta-z_0|<r,|J-J_0|<r\}.
\]
Let $\Theta_r:=\{\alpha+i\beta\in\cc^+ : |\alpha+i\beta-z_0|\leq r\}$ and let $\OO_r=\OO_{\Theta_r}$ be the circular neighborhood of $y$ in $\hh$ defined by $\Theta_r$. We now proceed as in the proof of \cite[Theorem~3.9]{TwistorJEMS}. Since $y\not\in W_f$, the normal function $N(f-f(y))$ is not identically vanishing. By Proposition~\ref{pro:Delta}, there exist $n\ge2$ and $g\in\sr_\rr(\OO)$ such that $N(f-f(y))=\Delta_y^n\,g$ and $V(g)\cap\s_y=\emptyset$. After choosing a smaller circular open domain containing $y$, we can suppose that $V(g)\cap\OO=\emptyset$ and that $f'_s\ne0$ on $\OO\setminus\rr$. Choose $r_0\in(0,\beta_0)$ sufficiently small to have $\OO_{r_0}\subset\OO$. Set $M:=\max_{z\in\Theta_{r_0}}\left(|F_2(z)|^{-1}\right)>0$. Let $r'\in(0,r_0]$ be such that 
\[
|F_1(z)-F_1(z')|+|F_2(z)-F_2(z')|\le {r_0}/M
\]
for every $z,z'\in\Theta_{r'}$. Let $\OO^+_{J_0}:=\OO\cap\cc^+_{J_0}$. By Hurwitz's Theorem (see e.g.~\cite[\S1.4]{Schiff}) applied to the holomorphic function $N(f-f(y))|_{\OO^+_{J_0}}$,  we can find a positive $r\le r'$ such that, for every $y_1\in V_r$, the function $N(f-f(y_1))|_{\OO^+_{J_0}}$ has exactly $n$ zeros in $\OO_{r'}\cap\OO^+_{J_0}$, counted with their multiplicities. 
Let $y_1=\alpha_1+J_1\beta_1\in V_r$ and $z_1:=\alpha_1+i\beta_1\in\Theta_r$.
The zero set $V(N(f-f(y_1)))\cap \OO_{r'}$ is the union of $h$ disjoint spheres $\s_1,\ldots, \s_h$, while  $V(f-f(y_1))\cap \OO_{r'}=\{y_1,\ldots,y_h\}$, where $y_k\in\s_k$ is a non-spherical zero of $f-f(y_1)$ of total multiplicity $m_k$ for each $k=1,\ldots,h$, with $\sum_k m_k=n$.
Since $y_k=\alpha_k+J_k\beta_k\in\OO_{r'}$ for $k=2,\ldots,h$, then $z_k:=\alpha_k+i\beta_k$ belongs to $\Theta_{r'}$. Moreover, since $F_1(z_k)+J_kF_2(z_k)=f(y_k)=f(y_1)=F_1(z_1)+J_1F_2(z_1)$ for every $k=2,\ldots,h$, it holds
\[J_k-J_1=\left(F_1(z_1)-F_1(z_k)+J_1(F_2(z_1)-F_2(z_k))\right)F_2(z_k)^{-1}.
\]
Therefore $|J_k-J_0|\le|J_k-J_1|+|J_1-J_0|< r_0+r$ for every $k=2,\ldots,h$. We can then set $V:=V_r$ and $V':=\{\alpha+J\beta\in\hh \,:\, \alpha,\beta\in\rr, J\in\s_\hh, |\alpha+i\beta-z_0|<r',|J-J_0|<r_0+r\}\supset V$. If~$r_0$ is sufficiently small, we get the required inclusion $V'\subset U$. 

If $y\in N_f\cap\rr$, the thesis follows directly from \cite[Theorem~3.9]{TwistorJEMS}.
\end{proof}

Let $f\in\sr(\OO)$ and let $y\in\OO$. Recall that $f$ is said to be a \emph{local homeomorphism} at $y$ if there exists an open neighborhood $U$ of $y$ in $\OO$ such that $f(U)$ is open in $\hh$ and the restriction of $f$ from $U$ to $f(U)$ is a homeomorphism.  
Let $B_f$ denote the \emph{branch set} of $f$, the set of points of $\OO$ at which $f$ fails to be a local homeomorphism. Evidently, $\Cl(D_f)\cup W_f\subset B_f$ and, by the Implicit Function Theorem, $\OO\setminus N_f\subset\OO\setminus B_f$. Consequently, it holds:
\[
\Cl(D_f)\cup W_f\subset B_f\subset N_f \quad \text{for every $f\in\sr(\OO)$}.
\]

\begin{theorem}\label{thm:branch}
Let $f\in\sr(\OO)$. Then $N_f=B_f$. More precisely, if $U$ is a non-empty open subset of $\OO$ such that the restriction $f|_U$ is injective, then $N_f\cap U=\emptyset$. In particular, $f$ is locally injective if and only if $N_f=\emptyset$.
\end{theorem}
\begin{proof}
If $f\in\SC(\OO)$, then the statement is evident, because $N_f=B_f=\OO$. Let $f\in\sr(\OO)\setminus\SC(\OO)$. Let $U$ be a non-empty open subset of $\OO$ such that the restriction $f|_U$ is injective. Suppose $N_f\cap U\neq\emptyset$ and choose $y\in N_f\cap U$. Since $f$ is injective locally at $y$, it follows that $y\not\in\Cl(D_f)\cup W_f$. Let $V,V'$ be the neighborhoods of $y$ with $V\subset V'\subset U$ given in the statement of Proposition~\ref{pro:branch}. By point {\it 1} of Theorem \ref{thm:Nf}, $V\not\subset N_f$. Choose $x\in V\setminus N_f$. From Proposition~\ref{pro:Delta} it follows that $x$ has total multiplicity $1$ as zero of $f-f(x)$. Then the fiber $f^{-1}(f(x))$ contains at least two distinct points in $V'$, which is a contradiction.
\end{proof}

We now apply the preceding results to study the local dimension of the singular set $N_f$. 

First we need to recall some basic definitions and facts concerning the dimension of a real analytic set. Let $V$ be a non-empty open subset of some $\R^n$ and let $A\subset V$ be a real analytic set. 
Consider $y\in A$ and denote by $A_y$ the germ at $y$ of $A$. Let $A_y=\bigcup_{\nu=1}^kA_{y,\nu}$ be the decomposition of $A_y$ into its irreducible real analytic components and, for every $\nu\in\{1,\ldots,k\}$, let $p_\nu\in\nn$ be the dimension of the irreducible real analytic germ $A_{y,\nu}$, defined by means of Weierstrass' Preparation Theorem as in \cite[Proposition 2, p.~32]{narasimhan}. The \emph{local dimension $\dim_y(A)$ of $A$ at~$y$} is given by $\dim_y(A):=\max_{\nu\in\{1,\ldots,k\}}p_\nu$ and the \emph{dimension $\dim(A)$ of $A$} by $\dim(A):=\max_{y\in A}\dim_y(A)$, see \cite[Definition 3, p.~40]{narasimhan}. If $A\neq\emptyset$ then each local dimension $\dim_y(A)$ and the dimension $\dim(A)$ of $A$ are natural numbers $\leq n$. Moreover, $\dim_y(A)=\dim(A\cap U')$ for every sufficiently small open neighborhood $U'$ of $y$ in $V$. For convention, we set $\dim_y(A):=-1$ if $y\in V\setminus A$ and $\dim(A):=-1$ if $A=\emptyset$. Hence $\dim_y(A)=-1$ if and only if $y\in V\setminus A$, and $\dim(A)=-1$ if and only if $A=\emptyset$.

We remark that, since every real analytic set is triangulable (see \cite{lojasiewicz}), the dimension of $A$ as a real analytic subset of $V$, the one recalled above, coincides with the topological dimension of $A$ as an arbitrary subset of $V$ (see \cite{hurewicz-wallman}).

The latter fact is important here. Indeed, in the proof of Theorem~\ref{thm:dimension} below, we will apply a result of Church \cite[Corollary~2.3]{Church}, that states the following: \emph{any light and open map $f:\R^n\rightarrow\R^n$ of class $\mscr{C}^n$ ($n\ge2$) has empty branch set $B_f$ or the topological dimension of $B_f$ is equal to $n-2$}. This result holds also locally. In particular, when $f\in\sr(\OO)\setminus\SC(\OO)$, 
it holds: $N_f=B_f$ (Theorem \ref{thm:branch}), the restriction of $f$ to $\OO\setminus(\Cl(D_f)\cup W_f)$ is light and open (Corollary \ref{cor:fibers} and Theorem \ref{thm:open}) and, given any $y\in N_f\setminus(\Cl(D_f)\cup W_f)$, the local dimension $\dim_y(N_f)$ coincides with the `topological dimension' of a sufficiently small open neighborhood of $y$ in $B_f$. Consequently, the mentioned result of Church implies that $\dim_y(N_f)=4-2=2$.

Let $f\in \sr(\OO)$. By point {\it 1} of Theorem \ref{thm:Nf}, $f\in\SC(\OO)$ if and only if $\dim_y(N_f)=4$ for some (or, equivalently, for every) $y\in\OO$. In particular, if $f\in\SC(\OO)$ then $\dim(N_f)=4$. The next result deals with the case $f\not\in\SC(\OO)$.


\begin{theorem}\label{thm:dimension}
Let $f\in\sr(\OO)\setminus\SC(\OO)$. Then the local dimensions $\dim_y(N_f)$, $\dim_y(D_f)$ and $\dim_y(W_f)$ belong to $\{-1,2,3\}$ and the local dimension $\dim_y(N_f\setminus(\Cl(D_f)\cup W_f))$ to $\{-1,2\}$ for every $y\in\OO$. In particular, the dimensions $\dim(N_f)$, $\dim(D_f)$ and $\dim(W_f)$ belong to $\{-1,2,3\}$, and the dimension $\dim(N_f\setminus(\Cl(D_f)\cup W_f))$ to $\{-1,2\}$.

More precisely, we have:
\begin{enumerate}
 \item If $\OO$ is a slice domain, then $N_f=\Cl(D_f)\cup(N_f\setminus\Cl(D_f))$ and if $N_f\ne\emptyset$ then one of the following holds:
  \begin{enumerate}
   \item[1.1.] $D_f=\emptyset$ and $\dim(N_f)=2$.
   \item[1.2.] $\dim_y(D_f)\in\{2,3\}$ for every $y\in D_f\neq\emptyset$ and
   $N_f=\Cl(D_f)$ or $\dim_y(N_f\setminus\Cl(D_f))=2$ for every $y\in N_f\setminus\Cl(D_f)\neq\emptyset$.
  \end{enumerate}
 \item If $\OO$ is a product domain, then $N_f=D_f\cup W_f\cup\left(N_f\setminus\left({D_f}\cup W_f\right)\right)$ and one of the following holds:
  \begin{enumerate}
   \item[2.1.] $D_f=\emptyset$, $W_f=\emptyset$ or $\dim_y(W_f)\in\{2,3\}$ for every $y\in W_f\neq\emptyset$, and $N_f\setminus({D_f}\cup W_f)=\emptyset$ or $\dim_y(N_f\setminus({D_f}\cup W_f))=2$ for every $y\in N_f\setminus({D_f}\cup W_f)\neq\emptyset$.
   \item[2.2.] $\dim_y(D_f)=2$ for every $y\in D_f\neq\emptyset$, $W_f=\emptyset$ or $\dim_y(W_f)=2$ for every $y\in W_f\neq\emptyset$, and $N_f\setminus({D_f}\cup W_f)=\emptyset$ or $\dim_y(N_f\setminus({D_f}\cup W_f))=2$ for every $y\in N_f\setminus({D_f}\cup W_f)\neq\emptyset$. However the case `$\,\dim_y(D_f)=2$ for every $y\in D_f\neq\emptyset$, $W_f=\emptyset$ and $N_f\setminus({D_f}\cup W_f)=\emptyset$' cannot occur.
 \item[2.3.] $\dim(D_f)=3$, $W_f=\emptyset$, and $N_f\setminus({D_f}\cup W_f)=\emptyset$ or $\dim_y(N_f\setminus({D_f}\cup W_f))=2$ for every $y\in N_f\setminus({D_f}\cup W_f)\neq\emptyset$.
  \end{enumerate}
\end{enumerate}
\end{theorem}
\begin{proof}
Let $f\in\sr(\OO)\setminus\SC(\OO)$ and let $y\in N_f$. Since the function $F_2:D\to\hh$ is real analytic and not locally constant, and $D_f$ is the circularization of the zero set $V(F_2)$ of $F_2$, we have that the local dimensions of $V(F_2)$ are either $-1$ or $0$ or $1$. Consequently, $\dim_y(D_f)\in\{2,3\}$ if $y\in D_f$. By Theorem \ref{thm:fibersformula} and \cite[Corollary~2.3]{Church}, $\dim_y(W_f)\in\{2,3\}$ if $y\in W_f$ and $\dim(N_f\setminus(\Cl(D_f)\cup W_f))=2$ if $y\in N_f\setminus(\Cl(D_f)\cup W_f)$. Point {\it 1} follows immediately from the fact that $W_f=\emptyset$ if $\OO$ is a slice domain. Suppose $\OO$ is a product domain. If $D_f\neq\emptyset$, then $\dim(D_f)=2$ (or, equivalently, $\dim_y(D_f)=2$ for every $y\in D_f$) or $\dim(D_f)=3$. Moreover $W_f=\emptyset$ or $W_f$ consists of a single wing. In the latter case Proposition \ref{prop:wings} implies that $\dim(D_f)=2$. By Theorem \ref{thm:covering-space}, if $\dim(D_f)=2$ and $W_f=\emptyset$, then $N_f\setminus(D_f\cup W_f)\neq\emptyset$. This proves point {\it 2}.
\end{proof}

\begin{remark}
If in point {\it 2.2} of the preceding statement $D_f\neq\emptyset$ and $W_f\neq\emptyset$, then $D_f$ is a union of isolated spheres $\s_x$, $W_f$ is a single wing, the intersection $D_f\cap W_f$ consists of isolated points and the Jacobian matrix $J_f(y)$ is null, that is $\cd(y)=f_s'(y)=0$ for every $y \in D_f\cap W_f$. This follows immediately from Proposition \ref{prop:wings}, equality \eqref{eq:matrix} and the transversality in $\hh$ between $D_f$ and $W_f$.
\end{remark}

\begin{corollary}
Let $f\in\sr(\OO)\setminus\SC(\OO)$. The following holds:
\begin{enumerate}
 \item If $y\in N_f$ and $(N_f)_y=\bigcup_{\nu=1}^k(N_f)_{y,\nu}$ is the decomposition of $(N_f)_y$ into its irreducible real analytic components, then the dimension of each $(N_f)_{y,\nu}$ belongs to $\{2,3\}$.
 \item $N_f$ does not have isolated points.
 \item There does not exist any open subset $U$ of $\OO$ such that $N_f\cap U$ is homeomorphic to $\rr$.
\end{enumerate}

\end{corollary}
\begin{proof}
By Theorem \ref{thm:dimension}, $\dim_y(N_f)\in\{2,3\}$ if $y\in N_f$; hence $\dim_y(N_f)\not\in\{0,1\}$.
\end{proof}

All the dimensional configurations mentioned in the statement of Theorem \ref{thm:dimension} can happen.

\begin{examples}
Let $\OO:=\hh\setminus\rr$ and let $f\in\sr(\OO)\setminus\SC(\OO)$. Define
\[
d_f:=\dim(D_f), \quad w_f:=\dim(W_f)=\dim(W_f\setminus D_f), \quad m_f:=\dim(N_f\setminus(D_f\cup W_f)).
\]
By Theorem \ref{thm:dimension}.2, the triple $(d_f,w_f,m_f)$ can assume at most eleven values:
\begin{align*}
(d_f,w_f,m_f)\in\big\{ & (-1,-1,-1),(-1,-1,2),(-1,2,-1),(-1,2,2),(-1,3,-1),\\
&(-1,3,2),(2,-1,2),(2,2,-1),(2,2,2),(3,-1,-1),(3,-1,2)\big\}.
\end{align*}

We will give examples of $f\in\sr_{\cc_i}(\OO)\setminus\SC(\OO)$ in which each of these values is assumed.

First, we recall that $\eta$ denotes the function in $\SC(\hh\setminus\rr)$ defined by $\eta(x):=\frac{1}{2}(1-I_xi)$, where $I_x:=\frac{\im(x)}{|\im(x)|}$. Observe that $\eta(x)=1$ if $x\in\cc_i^+$, $\eta(x)=0$ if $x\in\cc_{-i}^+$, $\eta^c(x)=\frac{1}{2}(1+I_xi)$, $\eta^c(x)=0$ if $x\in\cc_i^+$ and $\eta^c(x)=1$ if $x\in\cc_{-i}^+$. 
Thanks to the representation formula, given any holomorphic function $g:\cc_i\setminus\rr\to\cc_i$, the unique slice regular function $f\in\sr_{\cc_i}(\hh\setminus\rr)$ such that $f|_{\cc_i\setminus\rr}=g$ can be written as follows:
\[
f(x)=\eta(x)g(z_x)+\eta^c(x)g(\overline{z_x}) \quad \text{for every $x\in\hh\setminus\rr$, where $z_x:=\re(x)+i|\mr{Im}(x)|$.}
\] 
Let us present our examples $f$, denoting $F=F_1+\ui F_2$ the stem function of $f$.

\begin{enumerate}
  \item Evidently, if $f(x):=x$ then $N_f=\emptyset$ and hence $d_f=w_f=m_f=-1$.
 
 \item Let $f(x):=x^2-2xi$ (i.e. $f=f_1$ as in Examples \ref{ex:wings}). We know that $D_f=W_f=\emptyset$. 
Note that $p=i$ is a point of $N_f$, because $\frac{\partial f}{\partial x}=2x-2i$. In particular, $N_f\setminus(D_f\cup W_f)\neq\emptyset$. Consequently, $(d_f,w_f,m_f)=(-1,-1,2)$. This example was studied in \cite[Section 6]{TwistorJEMS} to construct a new non-constant orthogonal complex structure on open subsets of $\hh$.

\item Let $f(x):=x\eta(x)$ (i.e. $f=f_2$ as in Examples \ref{ex:wings}). We know that $D_f=\emptyset$ and $W_f=\cc^+_{-i}$. By Corollary \ref{cor:Nf}, $N_f$ is equal to the set of solutions of the equations 
$\big\langle\textstyle\cd(x),{f'_s(x)}\big\rangle=\big\langle\textstyle\cd(x),\im(x){f'_s(x)}\big\rangle=0$. 
 By a direct computation, we obtain:
\[
\big\langle\textstyle\cd(x),{f'_s(x)}\big\rangle=\frac14\left(\frac{x_1+|\im(x)|}{|\im(x)|}\right)
\;\;\text{and}\;\;
\big\langle\textstyle\cd(x),\im(x){f'_s(x)}\big\rangle=\frac{x_0}4\left(\frac{x_1+|\im(x)|}{|\im(x)|^2}\right).
\]
It follows that $N_f=\cc^+_{-i}$. Consequently, $(d_f,w_f,m_f)=(-1,2,-1)$.

 \item 
 Let $f(x):=\eta(x)g(z_x)$, where $g:\cc^+_i\to\cc_i$ is defined by $g(z):=e^{z^2-2zi}$.  Note that $F_1(z)=\frac{1}{2}g(z)$ and $F_2(z)=-\frac{i}{2}g(z)$ if $z\in\cc^+$. Consequently, $\langle F_1,F_2\rangle=|F_2|^2-|F_1|^2\equiv0$. Hence $N(f-c)\equiv0$ with $c\in\hh$ if and only if $2\langle F_1,c\rangle-|c|^2=\langle F_2,c\rangle\equiv0$. Since $F_1$ and $F_2$ are $\cc_i$-valued, the latter equations imply $c=0$. This shows that $W_f=f^{-1}(0)=\cc^+_{-i}$. Since $g$ is nowhere zero, $D_f=\emptyset$. Observe that $\cd(x)=\eta(x)g'(z_x)$, where $g'$ is the complex derivative of $g$. Since $g'(i)=0$, $p=i$ is a point of $N_f\setminus(D_f\cup W_f)$. It follows that $(d_f,w_f,m_f)=(-1,2,2)$.

 \item Let $f(x):=\eta(x)z_x-\eta^c(x)\frac{1}{\,\overline{z_x}\,}$ (i.e. $f=f_3$ as in Examples \ref{ex:wings}).
 We know that $f^{-1}(c)=W_{f,c}$ if and only if $c\in C$, where $C:=\{q_2j+q_3k\in\hh\,:\,q_2,q_3\in\rr, q_2^2+q_3^2=1\}$. In particular, $D_f=\emptyset$. By a direct computation, we obtain:
\[
\big\langle\textstyle\cd(x),\im(x){f'_s(x)}\big\rangle=\frac{x_0}{|\im(x)|}\big\langle\textstyle\cd(x),{f'_s(x)}\big\rangle=
\frac{x_0}{|\im(x)|}\frac{(1+|x|^2)(x_1(|x|^2+1)+|\im(x)|(|x|^2-1))}{4|\im(x)||x|^4}.
\]
It follows that $f(N_f)\subset C$. Consequently, $N_f=f^{-1}(C)=W_f$ and hence $(d_f,w_f,m_f)=(-1,3,-1)$.

 \item 
Let $f(x):=\eta(x)e_x-\eta^c(x)\frac{1}{\,\overline{e_x}\,}$, where $e_x:=e^{z_x^2-2z_xi}$. By Proposition \ref{prop:schwarz}, we know that $W_f=f^{-1}(C)$, where $C$ is as in {\it 5}. In particular, $D_f=\emptyset$. Since $\cd(i)=0$ and $f(i)=e\not\in C$, $p=i$ is a point of $N_f\setminus(D_f\cup W_f)$. It follows that $(d_f,w_f,m_f)=(-1,3,2)$.

\item Let $f(x)=(x^2+4)(x^2-2xi-1)$. Evidently, $W_f=\emptyset$. By a direct computation, we easily see that $F_2(z)=0$ if and only if $z=\pm 2i$, so $D_f=\s_{2i}$.  Since $\cd(i)=0$ and $i\not\in\s_{2i}$, $p=i$ is a point of $N_f\setminus(D_f\cup W_f)$. Consequently, $(d_f,w_f,m_f)=(2,-1,2)$.

\item Let $f(x):=\eta(x)(z_x^2+1)=(x^2+1)\eta(x)$ (i.e. $f=f^*_2$ as in Examples \ref{ex:wings}). We know that $D_f=\s_\hh$ and $W_f=\cc^+_{-i}$. Using Corollary \ref{cor:Nf} again, we obtain
\[
\big\langle\textstyle\cd(x),{f'_s(x)}\big\rangle=\frac{(|x|^2-1)(x_1+|\im(x)|)}{2|\im(x)|}
\,\;\;\text{and}\;\;\,
\big\langle\textstyle\cd(x),\im(x){f'_s(x)}\big\rangle=\frac{x_0(|x|^2+1)(x_1+|\im(x)|)}{2|\im(x)|^2}.
\]
If follows that $N_f=\s_\hh\cup\cc^+_{-i}=D_f\cup W_f$. Consequently, $(d_f,w_f,m_f)=(2,2,-1)$. 

\item Let $f(x):=\eta(x)(z_x^2-2z_xi+3)$. Since $z_x^2-2z_xi+3=0$ if and only if $z_x=3i$, and $\cc^+_{-i}\subset W_{f,0}$, it follows that $D_f=\s_{3i}$ and $W_f=\cc^+_{-i}$. Observe that $\cd(x)=\eta(x)(2z_x-2i)$ and hence $\cd(i)=0$. Since $p=i$ is a point of $N_f\setminus(D_f\cup W_f)$, we deduce that $(d_f,w_f,m_f)=(2,2,2)$.
 
\item Let $f(x):=x^2$. Evidently, $W_f=\emptyset$. By a direct computation, it is immediate to verify that $N_f=D_f=\mr{Im}(\hh)$. Consequently, $(d_f,w_f,m_f)=(3,-1,-1)$.
 
 \item Let $f(x):=x^3+3x$. By a direct computation, we easily see that $D_f$ is the tridimensional hyperboloid $3\re(x)^2-|\im(x)|^2+3=0$, $W_f=\emptyset$ and $N_f\setminus (D_f\cup W_f)=\s_\HH$. Consequently, $(d_f,w_f,m_f)=(3,-1,2)$.

\end{enumerate}

We summarize the data of above examples in the following diagram, in which we add $n_f:=\dim(N_f)=\max\{d_f,w_f,m_f\}$. 

\begin{center}
\begin{tabular}{|r|l|r|r|r|l|r|}
\hline
& $f\in\sr_{\cc_i}(\OO)\setminus\SC(\OO)$, $\OO=\hh\setminus\rr$ & $d_f$ & $w_f$ & $m_f$ & $p\in N_f\setminus(D_f\cup W_f)$ & $n_f$ \\
\hline
1. & $f(x):=x$ & $-1$ & $-1$& $-1$ & $\mathrm{none}$ & $-1$ \\
2. & $f(x):=x^2-2xi$ & $-1$ & $-1$ & $2$ & $p=i$ & $2$ \\
3. & $f(x):=x\eta(x)$ & $-1$ & $2$ & $-1$ & $\mathrm{none}$ & $2$ \\
4. & $f(x):=\eta(x)e_x$, $e_x:=e^{z_x^2-2z_xi}$
 & $-1$ & $2$ & $2$ & $p=i$ & $2$ \\
5. & $f(x):=\eta(x)z_x-\eta^c(x)\frac{1}{\,\overline{z_x}\,}$ & $-1$ & $3$ & $-1$ & $\mathrm{none}$ & $3$ \\
6. & $f(x):=\eta(x)e_x-\eta^c(x)\frac{1}{\,\overline{e_x}\,}$ & $-1$ & $3$ & $2$ & $p=i$ & $3$ \\
7. 
 & $f(x):=(x^2+4)(x^2-2xi-1)$ & $2$ & $-1$ & $2$ & $p=i$ & $2$ \\
8. & $f(x):=\eta(x)(z_x^2+1)$ & $2$ & $2$ & $-1$ & $\mathrm{none}$ & $2$ \\
9. & $f(x):=\eta(x)(z_x^2-2z_xi+3)$ & $2$ & $2$ & $2$ & $p=i$ & $2$ \\
10. & $f(x):=x^2$ & $3$ & $-1$ & $-1$ & $\mathrm{none}$ & $3$ \\
11. & $f(x):=x^3+3x$ & $3$ & $-1$ & $2$ &  $p=i$ & $3$ \\
\hline
\end{tabular}\vskip 8pt
\end{center}

Note that slice regular functions $f\in\sr(\OO)$ defined in the preceding examples 1, 2, 7, 10 and 11 extend to slice regular functions $f\in\sr(\HH)$, which give all the five possible values of $(d_f,w_f,m_f)$ predicted in Theorem \ref{thm:dimension}.1.
\end{examples}

\begin{remark}
If $f\in\widetilde\sr_\rr(\OO)\setminus\SC(\OO)$,
then $\dim_y(D_f)\in\{-1,3\}$ for every $y\in\OO=\OO_D$. Indeed, without loss of generality we can assume that $f=\I(F_1+\ui F_2)\in\sr_\rr(\OO)\setminus\SC(\OO)$. Then the degenerate set $D_f$ is the circularization of the zero set $V(F_2)$ of the (real-valued) real analytic function $F_2$. Since $F_2$ is harmonic and non-constant, $V(F_2)$ is empty or it is a real analytic curve of $D$ without isolated points. By Proposition \ref{pro:realfibers} and Theorem \ref{thm:dimension}, we know that, if $f\in\widetilde\sr_\rr(\OO)\setminus\SC(\OO)$, then
\[
(d_f,w_f,m_f)\in\{(-1,-1,-1),(-1,-1,2),(3,-1,-1),(3,-1,2)\}.
\]
All these four dimensional configurations can happen. For example, if $\OO=\HH\setminus\rr$, the va\-lues $(-1,-1,-1)$, $(3,-1,-1)$ and $(3,-1,2)$ are assumed for the above-mentioned slice functions $f(x):=x$, $f(x):=x^2$ and $f(x):=x^3+3x$, respectively. 
Let $f:\hh\setminus\{0\}\to\hh$ be the slice regular function $f(x):=x-x^{-1}$. Note that $f'_s(x)=1+|x|^{-2}$, so $D_f=\emptyset$; $W_f=\emptyset$ as well, because $\hh\setminus\{0\}$ is a slice domain.  Since $\cd(i)=0$, we have that $(d_f,w_f,m_f)=(-1,-1,2)$. 

\end{remark}


\section{A boundary univalence criterion} \label{sec:univalence}
We conclude this work by presenting one more result coming from the sign property of the Jacobian of a slice regular function $f\in\sr(\OO)$ and from the lightness of $f$ away from $\Cl(D_f)\cup W_f$. It extends to four dimensions a classical univalence theorem (see e.g.~\cite[Lemma~1.1]{Pommerenke}), which states that if $f$ is holomorphic on an open  neighborhood of a closed disc $D$ and injective on the boundary of $D$, then $f$ is injective on the whole $D$.

Given any subset $S$ of $\hh$, we denote by $\Cl_\hh(S)$ and $\partial_\hh U$ the closure and the boundary of $U$ in $\hh$, respectively. Recall that, if $S\subset\OO$, $\Cl(S)$ denotes the closure of $S$ in $\OO$ and hence $\Cl(S)=\Cl_\hh(S)\cap\OO$.

\begin{theorem}\label{thm:univalent}
Let $f\in\sr(\OO)$ and let $U$ be a non-empty bounded connected open subset of $\OO$ (not necessarily circular) such that $\Cl_\hh(U)\subset\OO\setminus(\Cl(D_f)\cup W_f)$ and $f(\partial_\hh U)=\partial_\hh f(U)$. If $f$ is injective on $\partial_\hh U$, then $f$ is injective on $\Cl_\hh(U)$ and $N_f\cap U=\emptyset$. 
\end{theorem}
\begin{proof}
Since $\emptyset\neq\Cl_\hh(U)\subset\OO':=\OO\setminus(\Cl(D_f)\cup W_f)$, point {\it 1} of Theorem \ref{thm:Nf} implies that $f\not\in\SC(\OO)$. Now, thanks to Theorem \ref{thm:open}, $f(U)$ is open in $\hh$ and hence $f(U)\cap\partial_\hh f(U)=\emptyset$. Let $d$ denote the Brouwer degree of $f|_U$ on the connected component of $\hh\setminus f(\partial_\hh U)=\hh\setminus \partial_\hh f(U)$ containing the connected set $f(U)$. Since for regular values $y\in f(U)$ of $f|_U$ it holds $\det(J_f(x))>0$ for every $x\in f^{-1}(y)\cap U$, $d$ is equal to the cardinality of $f^{-1}(y)\cap U$, whence $d\ge1$. 

Let us show that $d=1$. Suppose on the contrary that $d\geq2$. Denote by $g:\Cl_\hh(U)\to\hh$ the restriction of $f$ to the compact subset $\Cl_\hh(U)$ of $\OO'$. Corollary \ref{cor:fibers} implies that $g$ has finite fibers. From Theorem~\ref{thm:dimension}, it follows that $\dim(N_f\cap\OO')\le2$. Moreover, thanks to Corollary 2 and Theorem VI 7 of \cite[p.~46 and pp.~91-92]{hurewicz-wallman}, we deduce that the topological dimension of $\partial_\hh U$ is equal to $3$ and the compact set $N^*:=g^{-1}(g(N_f\cap\Cl_\hh(U)))$ has topological dimension $\leq2$. 
In particular, $(\partial_\hh U)\setminus N^*\neq\emptyset$. Choose $x\in(\partial_\hh U)\setminus N^*\subset(\partial_\hh U)\setminus N_f$ and an open neighborhood $V$ of $x$ in $\OO'$ such that $f|_V$ is injective. 
Let $\{x_n\}_n$ be a sequence in $(U\cap V)\setminus N^*$ converging to $x$. Since $d\geq2$, $f|_V$ is injective and each value $f(x_n)$ is a regular value of $f|_U$, there exists $x'_n\in U\setminus V$ such that $f(x'_n)=f(x_n)$. Extracting a subsequence if necessary, we can assume that $\{x'_n\}_n$ converges to some $x'\in\Cl_\hh(U)\setminus V$. It follows that $f(x')=f(x)\in f(\partial_\hh U)$. If $x'\in U$ then $f(x')\in f(U)\cap f(\partial_\hh U)=f(U)\cap\partial_\hh(f(U))=\emptyset$, which is impossible. Therefore $x'\in\partial_\hh U$, contradicting the injectivity of $f$ on $\partial_\hh U$. This proves that $d=1$.

Let us show that $f|_U$ is injective. Suppose on the contrary that there exist $p_1,p_2\in U$ such that $f(p_1)=f(p_2)$. Take disjoint neighborhoods $U_1$ of $p_1$ and $U_2$ of $p_2$ in $U$. By Theorem \ref{thm:open}, the sets $f(U_1)$ and $f(U_2)$ are open in $f(U)$. Consequently, $f(U_1)\cap f(U_2)$ is a non-empty open neighborhood of $f(p_1)$. Since the topological dimension of $g(N_f\cap\Cl_\hh(U))$ is $\leq2$, we have that $M^*:=(f(U_1)\cap f(U_2))\setminus g(N_f\cap\Cl_\hh(U))\neq\emptyset$. Fix $y\in M^*$. Observe that $y$ is a regular value of $f|_U$, $U_1\cap f^{-1}(y)\neq\emptyset$ and $U_2\cap f^{-1}(y)\neq\emptyset$. This implies that $d\geq2$, which is a contradiction. We have just proved that $f|_U$ is injective.

Since $f|_{\partial_\hh U}$ is injective and $f(\partial_\hh U)\cap f(U)=\emptyset$, it turns out that $f|_{\Cl_\hh(U)}$ is injective as well. The equality $N_f\cap U=\emptyset$ was proved in Theorem~\ref{thm:branch}.
\end{proof}

In the preceding statement, condition `$f(\partial_\hh U)=\partial_\hh f(U)$' cannot be omitted. Indeed, in our next and last example, we give a slice regular function $f:\hh\setminus\{0\}\to\hh$ and a non-empty circular bounded connected open subset $U$ of $\hh\setminus\{0\}$ such that $D_f=W_f=\emptyset$, $\mr{Cl}_\hh(U)\subset\hh\setminus\{0\}$, $f(\partial_\hh U)\neq\partial_\hh f(U)$, $f$ is injective on $\partial_\hh U$, but $f$ is not injective on $U$.

\begin{example}
Let $f:\hh\setminus\{0\}\to\hh$ be the slice regular function $f(x):=x-x^{-1}$ and let $U:=\{y\in\hh:\frac{1}{3}<|y|<4\}$. Note that $f'_s(x)=1+|x|^{-2}$, so $D_f=\emptyset$; $W_f=\emptyset$ as well, because $(\hh\setminus\{0\})\cap\R\neq\emptyset$. It holds $N_f=\s$. It is also evident that $\mr{Cl}_\hh(U)\subset\hh\setminus\{0\}$.

Let us prove that $f(\partial_\hh U)\neq\partial_\hh f(U)$. For each $r>0$, define $S_r:=\{y\in\hh:|y|=r\}$. Note that $\partial_\hh U=S_{\frac{1}{3}}\cup S_4$. In this way, $\frac{8}{3}=f(-\frac{1}{3})\in f(S_{\frac{1}{3}})\subset f(\partial_\hh U)$. On the other hand, $3$ is a point of $U$, $f(3)=\frac{8}{3}$, $\det(J_f(3))=|1+3^{-2}|^4\neq0$ by Theorem~\ref{thm:Jacobian}, and hence $\frac{8}{3}\not\in\partial_\hh f(U)$.

Let us show that $f$ is injective on $\partial_\hh U$. First, note that $f(S_{\frac{1}{3}})\cap f(S_4)=\emptyset$. Indeed, if $v\in S_{\frac{1}{3}}$ and $w\in S_4$, it holds:
\[
\textstyle
|f(v)|=|v^2-1||v|^{-1}\leq|v|+|v|^{-1}=\frac{10}{3}<\frac{15}{4}=|w|-|w|^{-1}\leq|w^2-1||w|^{-1}=|f(w)|.
\]
Let $r\in\{\frac{1}{3},4\}$. We have to show that $f$ is injective on $S_r$. Note that $f(x)=(|x|^2x-\overline{x})|x|^{-2}=(|x|^2-1)|x|^{-2}\mr{Re}(x)+(|x|^2+1)|x|^{-2}\mr{Im}(x)$ for each $x\in\hh\setminus\{0\}$; consequently, $f^{-1}(\R)=\R\setminus\{0\}$. Thanks to the latter equality and to the fact that $f$ is slice preserving, it suffices to prove that $f$ is injective on $S_r\cap\cc_i$. Define the function $f_r:[0,2\pi)\to\hh$ by $f_r(t):=f(r\cos(t)+ir\sin(t))=(r-r^{-1})\cos(t)+i(r+r^{-1})\sin(t)$. Since $r\neq1$, we have that $r-r^{-1}\neq0$; as a consequence, $f_r$ is injective. This proves the injectivity of $f$ on the whole $\partial_\hh U$.

The function $f$ is not injective on $U$; indeed, $2$ and $-\frac{1}{2}$ belong to $U$, and $f(2)=\frac{3}{2}=f(-\frac{1}{2})$.
\end{example}

\vspace{1em}

\noindent {\bf Acknowledgement.} This work was supported by GNSAGA of INdAM, and by the grants ``Progetto di Ricerca INdAM, Teoria delle funzioni ipercomplesse e applicazioni'', and PRIN ``Real and Complex Manifolds: Topology, Geometry and holomorphic dynamics'' of the Italian Ministry of Education.


\end{document}